\documentclass[12pt,reqno]{amsart}

\usepackage{verbatim}
\usepackage{amssymb}
\usepackage{fancyhdr}
\usepackage[pdftex]{hyperref}

\hypersetup{colorlinks=true, linkcolor=black, citecolor=black}

\newcommand{\Z}{{\mathbb Z}}

\newcommand{\cA}{{\mathcal A}}
\newcommand{\cC}{{\mathcal C}}
\newcommand{\cH}{{\mathcal H}}
\newcommand{\cP}{{\mathcal P}}

\newcommand{\E}{{\mathsf E}}

\newcommand{\wh}[1]{\widehat{#1}}

\newcommand{\hZn}{{\widehat{\Z_n}}}
\newcommand{\hA}{{\widehat A}}
\newcommand{\hS}{{\widehat S}}

\newcommand{\wt}[1]{\widetilde{#1}}

\newcommand{\alp}{\alpha}
\newcommand{\bet}{\beta}

\newcommand{\del}{\delta}
\newcommand{\eps}{\varepsilon}
\newcommand{\lam}{\lambda}
\newcommand{\ome}{\omega}
\newcommand{\sig}{\sigma}
\newcommand{\zet}{\zeta}
\renewcommand{\phi}{\varphi}

\newcommand{\Gam}{\Gamma}

\DeclareMathOperator{\ord}{ord}

\newcommand{\longc}{,\dotsc,}
\newcommand{\longp}{+\dotsb+}
\newcommand{\longge}{\ge\dotsb\ge}
\newcommand{\longu}{\cup\dotsb\cup}

\newcommand{\est}{\varnothing}
\newcommand{\sbs}{\subset}
\newcommand{\seq}{\subseteq}
\newcommand{\stm}{\setminus}

\newcommand{\<}{\langle}
\renewcommand{\>}{\rangle}

\newcommand{\lcl}{\left\lceil}
\newcommand{\rcl}{\right\rceil}

\newtheorem{claim}{Claim}
\newtheorem{theorem}{Theorem}[section]
\newtheorem{lemma}[theorem]{Lemma}
\newtheorem{corollary}[theorem]{Corollary}
\newtheorem{conjecture}[theorem]{Conjecture}
\newtheorem{proposition}[theorem]{Proposition}
\theoremstyle{remark}
\newtheorem{remark}{Remark}%[section]
\numberwithin{equation}{section}
\numberwithin{claim}{section}

\newcounter{case}
\newcounter{subcase}[case]
\newcommand{\newcases}{\setcounter{case}{0}\setcounter{subcase}{0}}
\newcommand{\case}[1]{%
  \stepcounter{case}%
  \medskip
  \noindent
  {\bf Case \arabic{case}: {#1}.}}
\newcommand{\subcase}[1]{%
  \stepcounter{subcase}%
  \medskip
  \noindent
  \emph{Case \arabic{case}.\arabic{subcase}: #1.}}

\newcommand{\refc}[1]{\ref{c:#1}}
\newcommand{\refj}[1]{\ref{j:#1}}
\newcommand{\refl}[1]{\ref{l:#1}}
\newcommand{\refm}[1]{\ref{m:#1}}
\newcommand{\reft}[1]{\ref{t:#1}}
\newcommand{\refp}[1]{\ref{p:#1}}
\newcommand{\refs}[1]{\ref{s:#1}}
\newcommand{\refb}[1]{\cite{b:#1}}
\newcommand{\refe}[1]{\eqref{e:#1}}

\newcommand{\mysectionname}{}
\newcommand{\customheader}{}
\newcommand{\mysection}[1]{\renewcommand{\mysectionname}{#1}\section{#1}}

\title{Small doubling in cyclic groups}

\author{Vsevolod F. Lev}
\email{seva@math.haifa.ac.il}

\address{Department of Mathematics, The University of Haifa at Oranim,
  Tivon 36006, Israel}


\begin{document}
\baselineskip=16pt

\begin{abstract}
We give a comprehensive description of the sets $A$ in finite cyclic groups
such that $|2A|<\frac94|A|$; namely, we show that any set with this
property is densely contained in a (one-dimensional) coset progression.
This improves earlier results of Deshouillers-Freiman and
Balasubramanian-Pandey.
\end{abstract}

\maketitle

\tableofcontents

\mysection{Introduction}\label{s:intro}

One of the central problems of the additive combinatorics is to understand
the structure of the small-doubling sets, or \emph{approximate subgroups},
which are sets $A$ of group elements such that the sumset
 $2A:=\{a+b\colon a,b\in A\}$ has size comparable with the size of $A$.

We use the additive notation throughout since we will be concerned with
abelian groups only, and particularly with the finite cyclic groups, which we
denote $\Z_n$; here $n$ is the order of the group. Our goal is to prove the
following result.

\begin{theorem}\label{t:main}
Let $n$ be a positive integer. If a set $A\seq\Z_n$ satisfies
$|2A|<\frac94|A|$, then one of the following holds:
\begin{itemize}
\item[i)] There is a subgroup $H\le \Z_n$ such that $A$ is contained in an
    $H$-coset and $|A|>C^{-1}|H|$, where
    $C=3\cdot10^4$. % C=28,\!875 would do
\item[ii)] There are a proper subgroup $H<\Z_n$ and an arithmetic
    progression $P$ of size $|P|>1$ such that $A\seq P+H$ and
       $$ (|P|-1)|H|\le|2A|-|A|. $$
\item[iii)] There is a proper subgroup $H<\Z_n$ such that $A$ meets exactly
    three $H$-cosets, the cosets are not in an arithmetic progression, and
      $$ 3|H|\le |2A|-|A|. $$
\end{itemize}
\end{theorem}

We notice that the coefficient $\frac94$ in Theorem~\reft{main} is best
possible and indeed, the assumption $|2A|<\frac94\,|A|$ cannot be relaxed
even to $|2A|\le\frac94\,|A|$: for instance, if $n$ is large enough, and
$A=\{-1,0,1\}\cup\{a\}$ with $a\notin\{-3\longc 3\}$ and $2a\notin \{-2\longc
2\}$, then $|2A|=\frac94\,|A|$ while $A$ does not have the structure
described in Theorem~\reft{main}.

Theorem~\reft{main} improves the following result by Deshouillers and
Freiman.
\begin{theorem}[{\cite[Theorem~1]{b:df}}]\label{t:noDF}
Let $n$ be a positive integer. If a set $A\seq\Z_n$ satisfies $|2A|<2.04|A|$,
then one of the following holds:
\begin{itemize}
\item[i)] There is a subgroup $H\le\Z_n$ such that $A$ is contained in an
    $H$-coset and $|A|>10^{-9}|H|$.
\item[ii)] There is a proper subgroup $H<\Z_n$ and an arithmetic
    progression $P$ of size $|P|>1$ such that $A\seq P+H$ and
       $$ (|P|-1)|H|\le|2A|-|A|. $$
\item[iii)] There is a proper subgroup $H<\Z_n$ such that $A$ intersects
    exactly three $H$-cosets, the cosets are not in an arithmetic
    progression, and
      $$ 3|H|\le |2A|-|A|. $$
\end{itemize}
Moreover, in ii) and iii) there is an $H$-coset containing at least
$\frac23|H|$ elements of $A$.
\end{theorem}

\begin{remark}
In ii) and iii), we have $|2A|\ge |H|+|A|$, which establishes properness of
$H$ as an immediate consequence of the other assertions of the theorem. In
the same vein, the existence of an $H$-coset containing  at least
$\frac23|H|$ elements of $A$ is immediate in the case iii,) and also in the
case ii) provided that $|P|\ge 6$: say, in the latter case, letting
$\tau:=|2A|/|A|<\frac94$ and averaging over all $H$-cosets contained in
$P+H$, we get
\begin{multline*}
  \max\{|A\cap(z+H)|\colon z\in P\}
       \ge \frac{|A|}{|P|} = \frac{|2A|-|A|}{(\tau-1)|P|} \\
             \ge \frac{(|P|-1)|H|}{(\tau-1)|P|}
                 = \left( 1- \frac1{|P|} \right) \frac{1}{\tau-1}\,|H|
                          > \frac56\,\frac45\,|H| = \frac23|H|.
\end{multline*}
\end{remark}

A version of Theorem~\reft{noDF} was proved by Balasubramanian and
Pandey~\cite[Theorem~2]{b:bp} who have, essentially, improved the coefficient
from $2.04$ to $2.1$ under some extra assumptions.

Two other classical results which Theorems~\reft{main} and~\reft{noDF} are
worth comparing with are Kneser's theorem and Freiman's $(3n-3)$-theorem; see
Sections~\refs{kkr} and~\refs{aux} for the exact statements and the
references. Kneser's result classifies small-doubling sets in arbitrary
abelian groups, including sets densely contained in cosets, but requires the
doubling coefficient $|2A|/|A|$ to be smaller than $2$. The $(3n-3)$-theorem,
on the other hand, allows the doubling coefficient to be as large as
$3-o(1)$, but assumes the underlying group to be torsion-free; specifically,
it says that if $A$ is a finite subset of a torsion-free abelian group such
that $|2A|\le 3|A|-4$, then $A$ is contained in an arithmetic progression $P$
with $|P|-1\le|2A|-|A|$. Both Kneser's and Freiman's theorem are employed in
our argument.

The proof of Theorem~\reft{main} is inductive, and for the induction to go
through, we actually prove the following version of the theorem.

\begin{theorem}\label{t:aux}
Let $n$ be a positive integer. If a set $A\seq\Z_n$ is not contained in a
proper coset and satisfies $|2A|<\min\{\frac94|A|,n\}$, then one of the
following holds:
\begin{itemize}
\item[i)]  $|2A|-|A|\ge C_0^{-1}n$ where $C_0=2.4\cdot10^4$. %$C_0=23,\!100$.
\item[ii)] There are a proper subgroup $H<\Z_n$ and an arithmetic
    progression $P$ of size $|P|>1$ such that $A\seq P+H$ and
       $$ (|P|-1)|H|\le|2A|-|A|. $$
\item[iii)] There is a proper subgroup $H<\Z_n$ such that $A$ intersects
    exactly three $H$-cosets, the cosets are not in an arithmetic
    progression, and
      $$ 3|H|\le |2A|-|A|. $$
\end{itemize}
\end{theorem}

\begin{proof}[Deduction of Theorem~\reft{main} from Theorem~\reft{aux}]
Suppose that $A\seq\Z_n$ satisfies $|2A|<\frac94\,|A|$ and, without loss of
generality, assume also that $0\in A$ and $|A|\ge 2$. Let $L\le\Z_n$ be the
subgroup generated by $A$.

If $|2A|=|L|$, then $|A|>\frac49\,|2A|=\frac49\,|L|$; thus, $A$ has the
structure of Theorem~\reft{main}~i). Assuming now that $|2A|<|L|$, we apply
Theorem~\reft{aux} to the set $A$ with $L$ (instead of $\Z_n$) as the
underlying group, and consider two possible cases.

If $A\sbs L$ satisfies the inequality of Theorem~\reft{aux}~i), then
$C_0^{-1}|L|\le|2A|-|A|<\frac54|A|$, so that $|A|>(4/5C_0)|L|=C^{-1}|L|$;
this is case~i) of Theorem~\reft{main}.

On the other hand, it is clear that Theorem~\reft{aux}~ii) implies
Theorem~\reft{main}~ii), and similarly for Theorem~\reft{aux}~iii).
\end{proof}

We thus focus on the proof of Theorem~\reft{aux}; once it is completed,
Theorem~\reft{main} will follow.

As explained above, the coefficient $9/4$ of Theorem~\reft{main} cannot be
replaced with a larger one. However, it is plausible to expect that the
following can be true.

\begin{conjecture}\label{j:relax}
For any $\eps>0$ there exist positive constants $C_1(\eps)$ and $C_2(\eps)$
such that if $n$ is a positive integer, and $A\seq\Z_n$ satisfies
$|A|<(C_1(\eps))^{-1}n$ and $|2A|<(3-\eps)\,|A|$, then there are a subset
$P\seq\Z_n$ with $|2P|/|P|\le|2A|/|A|$, and a proper subgroup $H<\Z_n$ such
that $A\seq P+H$,
       $$ (|2P|-|P|)|H|\le|2A|-|A|, $$
and either $|P|\le C_2(\eps)$, or $P$ is an arithmetic progression.
\end{conjecture}

We remark that the inequality $|2P|/|P|\le|2A|/|A|$ follows in fact from the
other assertions:
\begin{multline*}
  |A|\,\left(\frac{|2P|}{|P|}-1\right)
          \le |P||H| \,\left(\frac{|2P|}{|P|}-1\right) \\
                = (|2P|-|P|)|H| \le |2A|-|A|
                      = |A|\,\left(\frac{|2A|}{|A|}-1\right).
\end{multline*}

Theorem~\reft{main} and Conjecture~\refj{relax} show that any set with the
small doubling coefficient is, essentially, obtained by ``lifting'' a
small-doubling set which is either nicely structured (an arithmetic
progression), or otherwise belongs to a finite collection of sporadic
examples.

Our argument follows the general line of reasoning introduced by Freiman
in~\refb{f1} and then pursued by other authors; namely, we use character sums
to conclude that small doubling leads to a biased distribution, and then use
the bias as a starting point for a combinatorial part of the proof. The
improvements come from a refinement in the character sums component, in the
spirit of~\refb{ls}; from replacing the main auxiliary result used in
Deshouillers-Freiman \cite[Theorem~2]{b:df} with its stronger
version~\cite[Theorem~2]{b:l1}, see Section~\refs{rectif}; and, finally, from
using an intricate combinatorial analysis.

The rest of the paper is structured as follows. In the next section we
introduce notation that will be used throughout and considered standard. In
Section~\refs{rectif} we prove Theorem~\reft{aux} in the special case where
the image of the small-doubling set under a suitable homomorphism is
\emph{rectifiable}; although this  case is of principal importance, the proof
is, essentially, just a reduction to~\cite[Theorem~2]{b:l1}. In
Section~\refs{kkr} we present Kneser's theorem and a relaxed version of
Kemperman's theorem. In Section~\refs{VSDS} we establish a number of
properties of the sets with a ``very small'' doubling coefficients, including
the asymmetric case. Some other general results on set addition in abelian
groups, mostly of combinatorial nature, are gathered in Section~\refs{aux}.
Section~\refs{minX} establishes a number of results about the minimal
counterexample set (which, as we eventually show, does not exist). Two more
results of this sort, Lemmas~\refl{twocoset}, and~\refl{threecosets}, show
that the minimum counterexample set, if exists, meets at least four cosets of
any subgroup, with the obvious exceptions. The two lemmas are singled out
into dedicated Sections~\refs{twocoset} and~\refs{threecosets}. Their proofs
are quite technical and some readers may prefer to skip the details and
proceed to Section~\refs{charsum} where the character sum component of the
argument is presented. The proof is completed in the final
Section~\refs{final}.

\mysection{Notation}

Let $G$ be an abelian group.
\subsection{Groups}

By $A+B$ we denote the Minkowski sum of the sets $A,B\seq G$; that is,
$A+B=\{a+b\colon a\in A,\ b\in B\}$.

For a subgroup $H\le G$, by $\phi_H$ we denote the canonical homomorphism
$G\to G/H$; thus, for instance, with $\Z$ denoting the group of integers, we
have $\Z_n=\phi_{n\Z}(\Z)$. For $g_1,g_2\in G$, we may occasionally write
$g_1\equiv g_2\pmod H$ as an alternative to $g_1-g_2\in H$, $g_1+H=g_2+H$, or
$\phi_H(g_1)=\phi_H(g_2)$.

The \emph{period} (or \emph{stabilizer}) of a subset $S\seq G$ is the
subgroup $\pi(S):=\{g\in G\colon S+g=S\}\le G$, and $S$ is \emph{periodic} or
\emph{aperiodic} according to whether $\pi(S)\neq\{0\}$ or $\pi(S)=\{0\}$.

The \emph{index} of a subgroup $H\le G$, denoted $[G:H]$, is the size of the
quotient group $G/H$; thus, if $G$ is finite, then $[G:H]=|G|/|H|$.

We say that a coset $g+H$ is \emph{determined} by a subset $A\seq G$ if the
intersection $A\cap(g+H)$ is nonempty. In this case we also say that $A$
\emph{meets}, or \emph{intersects}, $g+H$.

An \emph{involution} of $G$ is an element $g\in G$ of order $2$. Importantly,
a cyclic group has at most one involution.

A finite subset $A$ of an abelian group will be called a
\emph{very-small-doubling set} (VSDS for short) if $|A|<\frac32\,|A|$;
equivalently, if $A$ is contained in a finite coset with density exceeding
$2/3$, see Section~\refs{VSDS}.

\subsection{Progressions}

For an integer $N\ge 2$, an $N$-term arithmetic progression in $G$ with the
difference $d\in G$ and the initial term $g\in G$ is a subset of $G$ of the
form $P=\{g,g+d\longc g+(N-1)d\}$; thus, for instance, cosets of finite
nonzero subgroups are considered arithmetic progressions, while singletons
are not. A progression is \emph{primitive} if its difference generates $G$.

For real $u\le v$, by $[u,v]$ we denote both the set of all integers $z$
satisfying $u\le z\le v$, and the image of this set under the canonical
homomorphism $\phi_{n\Z}$ from the group of integers to the cyclic group
under consideration.

\subsection{Local isomorphism and rectification}

We say that a subset $S\seq G$ is \emph{rectifiable} if it is \emph{locally
isomorphic} (or \emph{Freiman-isomorphic}) to a set of integers; that is, if
there is a mapping $\lam\colon S\to\Z$ such that for any $s_1\longc s_4\in
S$, we have $s_1+s_2=s_3+s_4$ if and only if
$\lam(s_1)+\lam(s_2)=\lam(s_3)+\lam(s_4$). Taking $s_1=s_3$, we see that
$\lam$ is injective; hence, $|\lam(S)|=|S|$. It is equally easy to see that
$|2\lam(S)|=|2S|$.

If $d\in G$ is an element of order $N\ge 2$, then any arithmetic progression
with the difference $d$, and with at most $(N+1)/2$ terms, is rectifiable.
Indeed, this is the only kind of rectifiable sets that actually appear below.

\subsection{Regularity}

For an integer $k\ge 2$, we say that a set $A\seq\Z_n$ is \emph{$k$-regular}
if it has the structure of Theorem~\reft{aux}~ii) with a $k$-element
progression $P$, and that $A$ is \emph{singular} if it has the structure of
Theorem~\reft{aux}~iii). Thus, Theorem~\reft{aux} essentially says that any
small-doubling set $A\seq\Z_n$ which is not densely contained in a coset is
either regular or singular.

\mysection{Theorem~\reft{aux} for rectifiable sets}\label{s:rectif}

One of the key ingredients of our argument is the following refinement of
\cite[Theorem~2]{b:df}.

\begin{theorem}[{\cite[Theorem~2]{b:l0}}]\label{t:small}
Suppose that $F$ is a finite group, and that $A$ is a finite subset of the
group $G=\Z\times F$. Let $s$ be the number of elements of the image of $A$
under the projection $G\to\Z$ along $F$. If $|2A|<3\big(1-\frac1s\big)|A|$,
then there exist an arithmetic progression $P\seq G$ of size $|P|\ge 3$ and a
subgroup $H\le\{0\}\times F$ such that $|P+H|=|P||H|$, $A\seq P+H$, and
$(|P|-1)|H|\le|2A|-|A|$.
\end{theorem}

We remark that the equality $|P+H|=|P||H|$ (which is somewhat implicit
in~\refb{l1}) is an easy consequence of the other assertions, as  it follows
by considering the difference of $P$. The difference cannot be contained in
the subgroup $\{0\}\times F$, since in this case $P$, and therefore also
$P+H$ and $A\seq P+H$, would be contained in a coset of $\{0\}\times F$,
leading to $s=1$ and thus contradicting the assumption
$|2A|<3\big(1-\frac1s\big)|A|$. Thus, the difference is of infinite order,
and therefore the difference of any two distinct elements of $P$ is of
infinite order, too, and does not belong to the finite subgroup $H$.

The following result establishes Theorem~\reft{aux} in the special case where
the image of $A$ under a suitable homomorphism is sufficiently large and
rectifiable.
\begin{proposition}\label{p:combo}
Suppose that $n$ is a positive integer, $L\le\Z_n$ is a subgroup, and
$A\seq\Z_n$ is a subset with $\phi_L(A)$ rectifiable. If $|2A|<3(1-1/s)|A|$,
where $s=|\phi_L(A)|$, then there exist an arithmetic progression $P\seq\Z_n$
and a proper subgroup $H<\Z_n$ such that $A\seq P+H$, $|P+H|=|P||H|$, and
$(|P|-1)|H|\le|2A|-|A|$.
\end{proposition}

We close this section with the deduction of Proposition~\refp{combo} from
Theorem~\reft{small}.
\begin{proof}[Proof of Proposition~\refp{combo}]
Since $\phi_L(A)$ is rectifiable, there is a local isomorphism, say $\lam$,
from $\phi_L(A)$ to $\Z$, and then the mapping $\psi\colon A\to\Z\times\Z_n$
defined by
  $$ \psi(a) := (\lam\circ\phi_L(a),a),\ a\in A $$
is a local isomorphism between $A$ and its image in $\Z\times\Z_n$.
Consequently, the set $\psi(A)\seq\Z\times\Z_n$ satisfies $|\psi(A)|=|A|$ and
$|2\psi(A)|=|2A|$. As a result,
  $$ \frac{|2\psi(A)|}{|\psi(A)|}
                                = \frac{|2A|}{|A|} < 3\Big(1-\frac1s\Big). $$
On the other hand, the size of the projection of the set
$\psi(A)\seq\Z\times\Z_n$ onto the first component of the direct product is
$|\lam\circ\phi_L(A)|=|\phi_L(A)|=s$. Thus, we can apply Theorem~\reft{small}
to the set $\psi(A)$ to find an arithmetic progression $Q\seq\Z\times\Z_n$ of
size $|Q|\ge 3$ and a subgroup $K\le\{0\}\times\Z_n$ such that
\begin{gather}
  \psi(A)\seq Q+K \label{e:0805a}
  \intertext{and}
  (|Q|-1)|K|\le|2\psi(A)|-|\psi(A)|=|2A|-|A|; \label{e:0805b}
\end{gather}
moreover, the elements of $Q$ reside in pairwise distinct $K$-cosets, and $K$
is proper in $\{0\}\times\Z_n$ since otherwise we would have $|K|=n$ and then
  $$ 2n \le (|Q|-1)|K| \le |2A|-|A| < 2|A| \le 2n. $$

Denoting by $\ome$ the projection of $\Z\times\Z_n$ onto the second
coordinate, we let $H:=\ome(K)\le\Z_n$ and $P:=\ome(Q)\seq\Z_n$.
From~\refe{0805a} and~\refe{0805b}, and in view of $|P|\le|Q|$ and $|H|=|K|$,
we readily conclude that $A\seq P+H$ and $(|P|-1)|H|\le|2A|-|A|$; however, an
extra effort is needed to ensure that the elements of $P$ lie in pairwise
distinct $H$-cosets, and that $|P|>1$.

To address the former point, we write $Q=\{g,g+d\longc g+(N-1)d\}$ where
$N:=|Q|\ge 3$. For $0\le i<j\le N-1$, the elements $\ome(g+id),\ome(g+jd)\in
P$ are in the same $H$-coset if and only if $(i-j)\ome(d)\in H$; that is, if
and only if $i\equiv j\pmod{\ord(\ome(d))}$, where $\ord(\ome(d))$ is the
order of $\ome(d)$ in $\Z_n/H$. Moreover, in this case
$\ome(g+id)+H=\ome(g+jd)+H$. Thus, if $\ord(\ome(d))\ge N$, then all elements
of $P$ reside in distinct $H$-cosets, while if $\ord(\ome(d))<N$, then the
sum $P+H$ will not be affected if we replace $P$ with its sub-progression
$\ome(\{g+id\colon 0\le i<\ord(\ome(d))\})$.

It remains to show that $|P|>1$. To this end we notice that if $|P|=1$, then
$A$ is contained in an $H$-coset; as a result,
  $$ (|Q|-1)|K| \ge 2|K| = 2|H| \ge 2|A| > |2A|-|A| $$
contradicting~\refe{0805b}.
\end{proof}

We remark that the quantity $|\phi_L(A)|$ is the number of $L$-cosets
determined by $A$. The situation where this quantity is too small for
Theorem~\reft{small} to be applicable is much more difficult to deal with.

\mysection{Kneser's and Kemperman's theorems and related
results}\label{s:kkr}

Recall that the period of a subset $A$ of an abelian group $G$ is the
subgroup $\pi(A):=\{g\in G\colon A+g=A\}\le G$, and that $A$ is periodic if
$\pi(A)$ is nonzero.

The following fundamental result due to Kneser is heavily used in our
argument.
\begin{theorem}[Kneser \cite{b:kn1,b:kn2}; see also \cite{b:mann}]
 \label{t:kneser}
Let $A$ and $B$ be finite, non-empty subsets of an abelian group $G$ such
that
  $$ |A+B| \le |A|+|B|-1. $$
Then, writing $H:=\pi(A+B)$, we have
  $$ |A+B| = |A+H|+|B+H|-|H|. $$
\end{theorem}

Since, in the above notation, we have $|A+H|\ge |A|$ and $|B+H|\ge |B|$,
Theorem~\reft{kneser} shows that $|A+B|\ge |A|+|B|-|\pi(A+B)|$ holds for any
finite, nonempty subsets $A$ and $B$ of an abelian group.
\begin{corollary}%\label{c:kneser}
Let $A$ and $B$ be finite, non-empty subsets of an abelian group $G$. If
  $$ |A+B| < |A|+|B|-1, $$
then $A+B$ is periodic.
\end{corollary}

Theorem~\reft{kneser} along with the corollary just stated will be referred
to as \emph{Kneser's theorem}.

Kemperman's structure theorem \cite{b:k} deals with the equality case of
Kneser's theorem. Following Kemperman~\refb{k}, we say that a pair $(A,B)$ of
finite subsets of an abelian group $G$ is \emph{elementary} if at least one
of the following holds:
\begin{itemize}
\item[i)] $\min\{|A|,|B|\}=1$;
\item[ii)] $A$ and $B$ are arithmetic progressions sharing a common
    difference $d\in G$, the order of which in $G$ is at least $|A|+|B|-1$;
\item[iii)] $A=g_1+(H_1\cup\{0\})$ and $B=g_2-(H_2\cup\{0\})$, where
    $g_1,g_2\in G$ and $H_1,H_2$ are non-empty subsets of a subgroup $H\le
    G$ such that $H=H_1\cup H_2\cup\{0\}$ is a partition of $H$. Moreover,
    $c:=g_1+g_2$ is the only element of $A+B$ with a unique representation
    as $a+b$ with $a\in A$ and $b\in B$;
\item[iv)] $A=g_1+H_1$ and $B=g_2-H_2$, where $g_1,g_2\in G$ and $H_1,H_2$
    are non-empty, aperiodic subsets of a subgroup $H\le G$ such that
    $H=H_1\cup H_2$ is a partition of $H$. Moreover, every element of $A+B$
    has at least two representations as $a+b$ with $a\in A$ and $b\in B$.
\end{itemize}

The following theorem proved in~\refb{l0} is a simplified and relaxed version
of the main result of~\refb{k}.
\begin{theorem}[{\cite[Theorem 1]{b:l0}}]\label{t:kemp}
Let $A$ and $B$ be finite, non-empty subsets of an abelian group $G$,
satisfying $|A+B|\le|A|+|B|-1$. Suppose that either $A+B\neq G$, or there is
a group element with a unique representation as $a+b$ with $a\in A$ and $b\in
B$. Then there exists a finite, proper subgroup $H<G$ such that
\begin{itemize}
\item[(i)]  $|C+H|-|C|\le|H|-1$ with $C$ substituted by any of the sets
    $A,B$, and $A+B$;
\item[(ii)] $(\phi_H(A),\phi_H(B))$ is an elementary pair in the quotient
    group $G/H=\phi_H(G)$.
\end{itemize}
\end{theorem}

\mysection{The very-small-doubling property}\label{s:VSDS}

We say that a finite set $A$ in an abelian group is a very-small-doubling set
(abbreviated below as VSDS) if $|2A|<\frac32\,|A|$. Thus, for instance, any
coset, and in particular any singleton, is a VSDS, while a two-element set is
a VSDS if and only if it is a coset.

The following basic fact is an easy corollary from Kneser's theorem. Along
with its (much subtler) noncommutative version, it can be found in~\refb{f4}.
\begin{lemma}\label{l:1.5}
A finite set $A$ in an abelian group is a VSDS if and only if there is a
subgroup $H$ such that $A$ is contained in an $H$-coset and
$|A|>\frac23\,|H|$. Moreover, in this case  $H:=A-A$ and $2A$ is an
$H$-coset.
\end{lemma}

The following asymmetric version is due to Olson (who also was primarily
concerned with the noncommutative settings).
\begin{lemma}[{\cite[Theorem~1]{b:o}}]\label{l:Olson}
If $A$ and $B$ are finite subsets of an abelian group, then either
 $|A+B|\ge |A|+\frac12\,|B|$, or $B$ is contained in a coset of the period
$H:=\pi(A+B)$.
\end{lemma}

As a corollary of Olson's result, we obtain several conditions which ensure
that $|A+B|\ge|A|+\frac12\,|B|$.

\begin{comment}
\begin{corollary}\label{c:Ogeneral}
Suppose that $A$ and $B$ are finite, nonempty subsets of an abelian group. If
$|A+B|<|A|+\frac12\,|B|$, then letting $H:=\pi(A+B)$ we have
$|A+H|-|A|<\frac12\,|B|\le\frac12\,|H|$, $|A|+\frac12\,|B|>|H|$, and
$|A|>\frac12\,|H|$.
\end{corollary}

\begin{proof}
By Lemma~\refl{Olson} and in view of $A+H\seq A+B+H=A+B$,
  $$ |H|-|A| \le |A+H| - |A| \le |A+B|-|A| < \frac12\,|B| \le \frac12\,|H| $$
which proves all the nontrivial assertions of the corollary.
\end{proof}
\end{comment}

\begin{corollary}\label{c:0308a}
Suppose that $A$ and $B$ are finite subsets of an abelian group such that
$|A+B|<|A|+\frac12\,|B|$. Let $H:=\pi(A+B)$. If $|A|\le|B|$, then
$|B|>\frac23\,|H|$, as a result of which $H=B-B$, $2B$ is an $H$-coset, and
$B$ is a VSDS.
\end{corollary}

\begin{proof}
By Lemma~\refl{Olson}, $B$ is contained in an $H$-coset. On the other hand,
  $$ |H| \le |A+B| < |A|+\frac12\,|B| \le \frac32\,|B| $$
and the rest follows from Lemma~\refl{1.5}.
\end{proof}

\begin{lemma}\label{l:consol}
Suppose that $A$ and $B$ are finite, nonempty subsets of an abelian group,
and let $H:=\pi(A+B)$. If $|A+B|<2\min\{|B|,\frac34\,|A|\}$, then
$|A|>\frac23|H|$ and $|B|>\frac12\,|H|$; moreover, each of the sets $A$ and
$B$ is contained in an $H$-coset and, indeed, $A+B$ is an $H$-coset.
\end{lemma}

Although Lemma~\refl{consol} is essentially contained, for instance,
in~\cite[Propositions~2 and~3]{b:bp} and~\cite[Proposition~2.1]{b:df}, we
present a complete proof.

\begin{proof}
Since $2\min\{|B|,\frac34\,|A|\}\le |B|+\frac32\,|A|<|A|+|B|$, by Kneser's
theorem,
\begin{equation}\label{e:2005a}
  |A+H|+|B+H|-|H| = |A+B| < 2|B|
\end{equation}
and also
\begin{equation}\label{e:2005b}
  |A+H|+|B+H|-|H| = |A+B| < \frac32\,|A|.
\end{equation}
This readily gives $|B|>\frac12|H|$ and $|A|>\frac23|H|$.

Let $\alp:=|A+H|/|H|$ and $\bet:=|B+H|/|H|$. From~\refe{2005a} we get
$\alp+\bet-1<2\bet$; hence $\alp<\bet+1$ and therefore $\alp\le\bet$.
Similarly,~\refe{2005b} gives $\alp+\bet-1<\frac32\alp$, leading to
$\bet\le(\alp+1)/2$. Consequently, $\alp\le\bet\le(\alp+1)/2$, whence
$\alp=\bet=1$. This means that each of $A$ and $B$ is contained in an
$H$-coset, and then $A+B$ is an $H$-coset by the definition of $H$.
\end{proof}

\begin{corollary}\label{c:AVSDS}
Suppose that $A$ and $B$ are finite, nonempty subsets of an abelian group. If
$A$ is not a VSDS, then $|A+B|\ge 2\min\{|B|,\frac34\,|A|\}$.
\end{corollary}

\begin{comment}
The following result is a stability version of Lemma~\refl{1.5}.
\begin{lemma}\label{l:3/2=}
Suppose that $A$ is a finite nonempty subset of an abelian group. If
$|2A|=\frac32|A|$, then there is a subgroup $H$ such that either $A$ is a
union of two $H$-cosets, or $A$ is contained in an $H$-coset, $2A$ is an
$H$-coset, and $|A|=\frac23|H|$.
\end{lemma}

\begin{proof}
for $|A|=2$ the assertion is trivial; suppose thus that $|A|\ge 4$. Since
$|2A|=\frac32|A|<2|A|-1$, the period $H:=\pi(2A)$ is nonzero by Kneser's
theorem. Moreover,
\begin{equation}\label{e:2305a}
  2|A+H|-|H| = |2A| = \frac32\,|A| \le \frac32\,|A+H|
\end{equation}
implies $|A+H|\le 2|H|$, and then we actually have $|A+H|=|H|$ or
$|A+H|=2|H|$. In the former case, substituting to~\refe{2305a} we get
$|H|=\frac32\,|A|$; moreover, $|A+H|=|H|$ shows that $A$ is contained in an
$H$-coset. In the latter case $A$ is contained in a union of two $H$-cosets,
and substitution to~\refe{2305a} gives $|A|=2|H|$; thus, $A$ is a union of
two $H$-cosets.
\end{proof}
\end{comment}

\mysection{More auxiliary results}\label{s:aux}

In this section we present a number of auxiliary results used in the proof of
Theorem~\reft{aux}. Some of the results are well known, or even classical,
some are original.

\begin{lemma}\label{l:2107a}
Suppose that $K$ is a subgroup, and that $A$ and $B$ are finite subsets of an
abelian group such that $A$ is contained in a single $K$-coset and
$|A|\ge\frac12\,|K|$.
\begin{itemize}
\item[i)]  If $|B|>|K|-|A|$, then $|A+B|\ge|K|$.
\item[ii)] If $|B|>2(|K|-|A|)$, then either $B$ is also contained in a
    single $K$-coset, or $|A+B|\ge |A|+|K|$.
\end{itemize}
\end{lemma}

\begin{proof}
Write $B=B_1\longu B_k$ where $|B_1|\longge|B_k|>0$, the union is disjoint,
and each $B_i$ is contained in a single $K$-coset.

\smallskip
 i) If $k=1$, then $|A+B|=|K|$ by the pigeonhole principle; if $k\ge 2$, then
$|A+B|\ge k|A|\ge2|A|\ge|K|$.

\smallskip
 ii) If $k=2$, then $|B_1|\ge\frac12\,|B|>|K|-|A|$ whence
$|A+B|=|A+B_1|+|A+B_2|\ge |K|+|A|$. If $k\ge 3$, then
$|A+B|\ge3|A|\ge|K|+|A|$.
\end{proof}

\begin{lemma}\label{l:Mantel}
Suppose that $A$ is a finite subset of an abelian group. If $|2A|\le 3|A|-4$,
then there are at most $|A|^2/4$ group elements possessing a unique
representation as $a-b$ with $a,b\in A$.
\end{lemma}

\begin{proof}
Consider the graph $\Gam$ with $A$ as a vertex set, with the vertices
 $a,b\in A$ adjacent if and only if $a-b$ has a unique representation as a
difference of two elements of $A$. If $a,b,c\in A$ induce a triangle in
$\Gam$, then by the Bonferroni inequalities we have
  $$ |2A| \ge |(A+a)\cup(A+b)\cup(A+c)| \ge 3|A| - 3, $$
contradicting the assumptions. Thus, $\Gam$ is triangle-free, and by Mantel's
theorem (which can be found in most of the standard graph theory textbooks),
the number of edges of $\Gam$ is at most $|A|^2/2$. However, the edges of
$\Gam$ are in a one-to-one correspondence with the uniquely representable
elements.
\end{proof}

Freiman's classical result known as ``the $(3n-3)$-theorem'' can be stated as
follows.
\begin{theorem}[Freiman~\refb{f1}]\label{t:3n-3}
Suppose that $A$ is a finite, nonempty set of integers, and $l\ge 1$ is an
integer. If $A$ is not contained in an $l$-term arithmetic progression, then
$|2A|\ge\min\{l,2|A|-3\}+|A|$.
\end{theorem}

For a modern exposition of Theorem~\reft{3n-3} and related results see, for
instance,~\cite[Chapter~7]{b:g}, \cite[Theorem 1.13]{b:n},
or~\cite[Theorem~5.11]{b:tv}.

We need yet another well-know result of Freiman.
\begin{lemma}[Freiman~\refb{f3}]\label{l:Fpoints}
Suppose that $Z$ is a finite subset of the unit circle on the complex plane.
If
  $$ \Big| \sum_{z\in Z} z \Big| = \eta |Z|,\quad \eta\in[0,1], $$
then there is an open arc of the circle of the angle measure $\pi$ containing
at least $\frac12(1+\eta)|Z|$ elements of $Z$.
\end{lemma}

The following basic lemma shows that rectifiable sets cannot have a strong
correlation with finite cosets.
\begin{lemma}\label{l:rect2A}
If $A$ is a rectifiable subset of an abelian group $G$, then for any finite
subgroup $K\le G$ and any element $g\in G$ we have
$|A\cap(g+K)|\le\frac12(|K|+1)$.
\end{lemma}

\begin{proof}
Let $A_0:=(A-g)\cap K$. If $|A_0|>\frac12(|K|+1)$ then, by the pigeonhole
principle, $2A_0=K$ and moreover, any element of $K$ has at least two
representations as a sum of two elements of $A_0$. At the same time, for any
finite integer set $B$ with $|B|\ge 2$, there are at least two elements of
$2B$ possessing a unique representation as a sum of two elements of $B$.
Thus, $A_0$ is not rectifiable; hence, neither is $A$.
\end{proof}

\begin{comment}
\begin{corollary}\label{c:rect2A}
A VSDS is rectifiable if and only if it is a singleton.
\end{corollary}
\end{comment}

\begin{proof}
Suppose that $A$ is a VSDS, and let $K:=A-A$; thus, $K$ is a subgroup, $A$ is
contained in a $K$-coset, and $|A|>\frac23\,|K|$. If $A$ is rectifiable, then
by the lemma $|A|\le \frac12(|K|+1)$. Therefore $\frac23\,|K|<\frac12(|K|+1)$
implying $|K|\le 2$ and, consequently, $|A|=1$.
\end{proof}

\begin{lemma}\label{l:kemp}
Suppose that $A$ and $B$ are finite subsets of an abelian group $G$ such that
$|A+B|\le|A|+|B|-1$, $|A|+|B|\le|G|-1$, and $\min\{|A|,|B|\}\ge 2$. If $B$ is
rectifiable, not an arithmetic progression, and not contained in a proper
coset, then there is a nonzero, finite, proper subgroup $H<G$ such that $B$
meets two $H$-cosets and has exactly $\frac{|H|+1}2$ elements in each of
them.
\end{lemma}

\begin{proof}
In view of $|A+B|\le|A|+|B|-1<|A|+|B|<|G|$, we can apply Theorem~\reft{kemp}
to find a finite, proper subgroup $H<G$ such that $|B+H|\le|B|+|H|-1$ and
$(\phi_H(A),\phi_H(B))$ is an elementary pair in the quotient group $G/H$.
Denoting by $k$ the number of $H$-cosets determined by $B$, we have
$|B+H|=k|H|$ and then, by Lemma~\refl{rect2A},
  $$ k|H| \le \frac{|H|+1}2\,k + |H| - 1; $$
that is,
  $$ (k-2)(|H|-1)\le 0. $$
Thus, either $k\le 2$, or $H=\{0\}$. In the latter case $(A,B)$ is an
elementary pair in $G$; however, this option is ruled out by the assumptions
of the lemma. We cannot have $k=1$ either as $B$ is not contained in a proper
coset. Thus, $k=2$, and then $B$ meets two $H$-cosets and has exactly
$\frac{|H|+1}2$ elements in each of them.
\end{proof}

The following simple lemma classifies three-element subsets of abelian groups
with at most one involution.
\begin{lemma}\label{l:3cyclic}
If $A$ is a three-element subset of an abelian group possessing at most one
involution, then one of the following holds:
\begin{itemize}
\item[i)]  $A$ is a coset of a three-element subgroup; accordingly,
    $|2A|=3$.
\item[ii)] $A$ is a three-term arithmetic progression with the difference
    of order at least $4$, and either $|2A|=4$ (if the difference has order
    exactly $4$), or $|2A|=5$ (if the difference has order at least $5$).
\item[iii)] $A=\{a,a+d,b\}$ where $d$ is an involution, $2b\ne 2a+d$, and
    $|2A|=5$.
\item[iv)] $A$ is neither a coset nor an arithmetic progression, and
    $|2A|=6$.
\end{itemize}
\end{lemma}
We omit the somewhat technical, but straightforward proof.
% A hint: does $X$ contain a 3AP? If it does, is this 3AP a coset? What is
% the order of the difference of this 3AP? If it does not, are there any
% coinciding sums among the x_i+x_j 's?

\begin{lemma}\label{l:alpha}
Suppose that $\cA=\{\alp_1,\alp_2,\alp_3\}$ is a subset of an abelian group
such that all sums $\alp_i+\alp_j$ with $1\le i\le j\le 3$ are pairwise
distinct (as a result of which $\alp_1,\alp_2,\alp_3$ are pairwise distinct).
If there are indices $i,j,k,l\in\{1,2,3\}$ and a group element
$\bet\notin\cA$ such that $\bet=\alp_i+\alp_j-\alp_1=\alp_k+\alp_l-\alp_2$,
then either $\cA\cup\{\bet\}$ is a four-term arithmetic progression, or
$\{\alp_1,\alp_2,\bet\}$ is a coset of a $3$-element subgroup.
\end{lemma}

\begin{proof}
From $\alp_i+\alp_j-\alp_1\notin\cA$ we get $i,j\in\{2,3\}$ and from
$\alp_k+\alp_l-\alp_2\notin\cA$ we get $k,l\in\{1,3\}$. If $\{i,j\}$ share a
common element with $\{k,l\}$, then assuming for definiteness that this
element is $i=k$ we get $\alp_j-\alp_1=\alp_l-\alp_2$ and consequently
$\alp_j+\alp_2=\alp_l+\alp_1$, which is impossible in view of $j\ne 1$ and
$l\ne 2$. Thus, $\{i,j\}$ is disjoint from $\{k,l\}$, and without loss of
generality, we can assume that $k=l\notin\{i,j\}$.

If $i\ne j$, then $\{i,j\}=\{2,3\}$, $k=1$, and
$\bet=\alp_2+\alp_3-\alp_1=2\alp_1-\alp_2$, implying
$\alp_3+2\alp_2=3\alp_1$. Letting $d:=\alp_1-\alp_2$, we thus have
$\alp_1=\alp_2+d$, $\bet=\alp_2+2d$, and $\alp_3=\alp_2+3d$, showing that $d$
has order at least $3$, and
$\cA\cup\{\bet\}=\{\alp_2,\alp_2+d,\alp_2+2d,\alp_2+3d\}$.

Finally, if $i=j$, then either $i=3,\ k=1$, or $k=3,\ i=2$, or $i=2,\ k=1$.
In the first case we have $2\alp_3-\alp_1=2\alp_1-\alp_2$ leading to
$\alp_2+2\alp_3=3\alp_1$, in the second case we similarly have
$\alp_1+2\alp_3=3\alp_2$; up to a renumbering, these cases were considered
above. In the third case where $i=2$ and $k=1$, we get $3\alp_1=3\alp_2$ and
$\bet=2\alp_2-\alp_1$; that is, $\del:=\alp_2-\alp_1$ has order $3$, and we
have $\alp_2=\alp_1+\del$ and $\bet=\alp_1+2\del$.
\end{proof}

\mysection{Partial results and the minimal counterexample} \label{s:minX}

In this section, assuming that Theorem~\reft{aux} is wrong, we study the
properties of the minimal counterexample set.

\begin{lemma}\label{l:MofA}
Suppose that Theorem~\reft{aux} is wrong. If $A\seq\Z_n$ is a counterexample
with $n$ smallest possible, then $|2A+L|-|2A|>|A+L|-|A|$ holds for any
nonzero subgroup $L<\Z_n$ satisfying $2A+L\ne\Z_n$.
\end{lemma}

\begin{proof}
Suppose that $A\seq\Z_n$ is not contained in a proper coset and satisfies
$2\le|2A|<\min\{\frac94|A|,n\}$ (as a result of which $n\ge 3$), but none of
the conclusions of the theorem holds true.

Suppose also, for a contradiction, that $L\le\Z_n$ is a nonzero subgroup with
$|2A+L|-|2A|\le|A+L|-|A|$ and $2A+L\ne\Z_n$. Notice that the last condition
implies that $L$ is proper.

Write $\cA:=\phi_L(A)$. If we had $|\cA|=1$, then $A$ were contained in a
single $L$-coset; thus, $|\cA|\ge 2$. On the other hand, $2A+L\ne\Z_n$ shows
that $2\cA\ne\Z_n/L$. We also have
  $$ |2A+L| \le |A+L| + |2A| - |A|
                                < |A+L| + \frac54\,|A| \le \frac94\,|A+L|, $$
whence
  $$ |2\cA| = \frac{|2A+L|}{|L|} < \frac94\frac{|A+L|}{|L|} = \frac94|\cA|. $$
The minimality of $n$ shows now that the set $\cA\seq\Z_n/L$ is not a
counterexample to Theorem~\reft{aux}. This means that
 there is a subgroup $\cH\le\Z_n/L$ such that one
of the following holds:
\begin{itemize}
\item[i)]   $|2\cA|-|\cA|>C_0^{-1}|\Z_n/L|$.
\item[ii)]  There is an arithmetic progression $\cP\seq\Z_n/L$ with
    $\cA\seq\cP+\cH$ and
       $$ (|\cP|-1)|\cH|\le|2\cA|-|\cA|. $$
\item[iii)] $\cA$ meets exactly three $\cH$-cosets which are not in an
    arithmetic progression, and
      $$ 3|\cH|\le |2\cA|-|\cA|. $$
\end{itemize}
Let $H:=\phi_L^{-1}(\cH)\le\Z_n$.

In the case i), we have
  $$ |2A|-|A| \ge |2A+L|-|A+L| = (|2\cA|-|\cA|)|L| > C_0^{-1}n. $$

In the case ii), we define $\wt{c},\wt{d}\in\Z_n/L$ to be the initial term
and the difference of $\cP$. Choosing $c,d\in\Z_n$ with $\phi_L(c)=\wt{c}$
and $\phi_L(d)=\wt{d}$, and letting $P:=\{c,c+d\longc c+(|\cP|-1)d\}$, we get
a progression $P\seq\Z_n$ with $|P|=|\cP|$ and $\phi_L^{-1}(\cP)=P+L$. From
$\cA\seq\cP+\cH$ we derive then that $A\seq P+H$, and from
$(|\cP|-1)|\cH|\le|2\cA|-|\cA|$ we obtain
  $$ (|P|-1)|H| = (|\cP|-1)|\cH||L|
                     \le (|2\cA|-|\cA|)|L| = |2A+L|-|A+L| \le |2A|-|A|. $$

Finally, in the case iii) it is immediately seen that $A$ is contained in a
union of three $H$-cosets which are not in an arithmetic progression. Also,
  $$ 3|H| = 3|\cH||L| \le (|2\cA|-|\cA|)|L| = |2A+L| - |A+L| \le |2A|-|A|. $$

In any case, $A$ has the structure described in the theorem; hence, is not a
counterexample.
\end{proof}

\begin{lemma}\label{l:2Aper}
Suppose that Theorem~\reft{aux} is wrong. If $A\seq\Z_n$ is a counterexample
with $n$ smallest possible, then $2A$ is aperiodic.
\end{lemma}

\begin{proof}
Let $L:=\pi(2A)$. Observing that $2A+L=2A\ne\Z_n$, we apply
Lemma~\refl{MofA}. The inequality of the lemma is clearly violated, showing
that $L$ is the zero subgroup.
\end{proof}

\begin{lemma}\label{l:AuL}
Suppose that Theorem~\reft{aux} is wrong. If $A\seq\Z_n$ is a counterexample
with $n$ smallest possible, then $|A+L|\ge|A|+|L|$ holds for any nonzero,
proper subgroup $L<\Z_n$.
\end{lemma}

\begin{proof}
Since $A\seq\Z_n$ satisfies the assumptions of Theorem~\reft{aux}, it is not
contained in a proper coset, and $2\le|2A|<\min\{\frac94\,|A|,n\}$. Suppose
for a contradiction that, in addition, we also have
\begin{equation}\label{e:A+L<}
  |A+L|<|A|+|L|
\end{equation}
with $L<\Z_n$ nonzero and proper. Since $|2A|<n$ implies $|A|\le\frac12\,n$
by the pigeonhole principle, and since the properness of $L$ implies
$|L|\le\frac12\,n$, as a consequence of~\refe{A+L<} we have $|A+L|<n$. Thus,
there is an $L$-coset disjoint with $A$, and since $A$ is not contained in a
proper coset, we conclude that $|L|\le\frac13\,n$. Re-using~\refe{A+L<}, we
now get
\begin{equation}\label{e:ALnew}
  |A+L| < \frac56\,n.
\end{equation}

Consider the coset decomposition
  $$ A = (a_0+L_0)\cup(a_1+L_1)\longu(a_k+L_k), $$
where $L_0,L_1\longc L_k\seq L$ are nonempty, $a_0,a_1\longc a_k\in A$, and
$a_i\not\equiv a_j\pmod L$. Renumbering, we further assume that
$|L_0|\ge|L_1|\longge|L_k|>0$. From
  $$ (|L|-|L_0|) + (|L|-|L_1|) \longp(|L|-|L_k|) = |A+L|-|A| < |L| $$
we derive that $|L_i|+|L_j|>|L|$, and therefore
$(a_i+L_i)+(a_j+L_j)=a_i+a_j+L$ for all $i,j\in[0,k]$, with the only possible
exception of $i=j=0$.

As a result,
\begin{equation}\label{e:X3a}
  |2A+L|-|2A| = |L|-|2L_0| \le |L|-|L_0| \le |A+L|-|A|,
\end{equation}
and applying Lemma~\refl{MofA}, we conclude that $2A+L=\Z_n$. Substituting
this back to \refe{X3a} and using~\refe{ALnew}, we obtain
  $$ |2A|-|A| \ge n-|A+L| > \frac16\,n. $$
Therefore $A$ satisfies the condition of Theorem~\reft{aux}~i), a
contradiction.
\end{proof}

\begin{lemma}\label{l:kemp1}
Suppose that Theorem~\reft{aux} is wrong. If $A\seq\Z_n$ is a counterexample
with $n$ smallest possible, then for any subset $B\seq\Z_n$ with
 $|A|\ge|B|\ge 2$ we have $|A+B|\ge|A|+|B|$.
\end{lemma}

\begin{proof}
Suppose that $|A|\ge|B|\ge 2$ and $|A+B|<|A|+|B|$. Observing that these
assumptions along with $|A|\le\frac12\,n$ (following from $2A\ne\Z_n$) give
$|A+B|<n$, we apply Theorem~\reft{kemp}. By the theorem, there is a finite,
proper subgroup $L<\Z_n$ such that $|A+L|\le|A|+|L|-1$ and
$(\phi_L(A),\phi_L(B))$ is an elementary pair in the quotient group $\Z_n/L$.
By Lemma~\refl{AuL}, we have $L=\{0\}$; thus, $(A,B)$ is an elementary pair
in the original group $\Z_n$. Inspecting the list of elementary pairs from
Section~\refs{kkr}, we see that $(A,B)$ is neither type i) nor type ii) (if
$A$ were an arithmetic progression, it would be regular.) Thus, $(A,B)$ is
elementary of type iii) or iv). In each of these cases, there is a subgroup
$H\le\Z_n$ such that each of $A$ and $B$ is contained in an $H$-coset, and
$|A|+|B|\ge|H|$. Since $A$ is not contained in a proper coset, we actually
have $H=\Z_n$, and then $2|A|\ge|A|+|B|\ge n$ whence $|A|\ge\frac12\,n$.
Combined with the observation at the beginning of the proof, this gives
$|A|=\frac12\,n$.

On the other hand, since $2A$ is aperiodic (Lemma~\refl{2Aper}), by Kneser's
theorem we have $|2A|\ge2|A|-1$. Therefore $|2A|-|A|\ge|A|-1=\frac12\,n-1\ge
C_0^{-1}n$, the last estimate following from $n=2|A|\ge 4$. This shows that
$A$ satisfies the inequality of Theorem~\reft{aux}~i).
\end{proof}

\begin{lemma}\label{l:A'+A''}
Suppose that Theorem~\reft{aux} is wrong. If $A\seq\Z_n$ is a counterexample
with $n$ smallest possible, then for any pair of nonempty subsets
$A',A''\seq\Z_n$ with $A'\cup A''=A$, we have
$|A'+A''|\ge\min\{|A'|+|A''|-1,n\}$.
\end{lemma}

\begin{proof}
Assuming $|A'+A''|<|A'|+|A''|-1$ and $|A'+A''|<n$, let $L:=\pi(A'+A'')$.
Notice that $L$ is nonzero by Kneser's theorem, and that $L$ is proper as
otherwise we would have $|A'+A''|=n$.

Let $g_1\longc g_k$ be representatives of the $L$-cosets determined by $A$.
We have
\begin{align*}
  |A+L|-|A|
    &= \sum_{i=1}^k (|L|-|(g_i+L)\cap A|) \\
    &\le \sum_{\substack{1\le i\le k \\ (g_i+L)\cap A'\ne\est}}
                                                       (|L|-|(g_i+L)\cap A|)
            + \sum_{\substack{1\le i\le k \\ (g_i+L)\cap A''\ne\est}}
                                                    (|L|-|(g_i+L)\cap A|) \\
    &\le \sum_{\substack{1\le i\le k \\ (g_i+L)\cap A'\ne\est}}
                                                      (|L|-|(g_i+L)\cap A'|)
           + \sum_{\substack{1\le i\le k \\ (g_i+L)\cap A''\ne\est}}
                                                  (|L|-|(g_i+L)\cap A''|) \\
    &= (|A'+L|-|A'|) + (|A''+L|-|A''|).
\end{align*}
By Kneser's theorem and the assumption $|A'+A''|<|A'|+|A''|-1$, the
right-hand side is
  $$ |A'+A''| + |L| - |A'| - |A''| < |L|. $$
Thus, $|A+L|-|A|<|L|$, contradicting Lemma~\refl{AuL}.
\end{proof}

\begin{lemma}\label{l:numbers}
Suppose that Theorem~\reft{aux} is wrong. If $A\seq\Z_n$ is a counterexample
with $n$ smallest possible, then $4\le |A|\le C_0^{-1}n$ and $8\le|2A|\le
2C_0^{-1}n$.
\end{lemma}

\begin{proof}
Applying Lemma~\refl{kemp1} with $B=A$ we get $|2A|\ge2|A|$, resulting in
  $$ 2\le |A| \le |2A|-|A| \le C_0^{-1}n $$
and, consequently, in % $n\ge 46,\!200$ and
  $$ |2A| \le |A| + C_0^{-1}n \le 2C_0^{-1}n. $$
It remains to show that $|A|\ge 4$ and, therefore, $|2A|\ge 8$.

We thus have to treat the cases where $|A|=2$ and $|A|=3$. If $|A|=2$, then
(trivially) $|2A|\le 3$, contradicting Lemma~\refl{kemp1} (applied with
$B=A$). If $|A|=3$, then $|2A|\ge 6$ by Lemma~\refl{kemp1} and therefore $A$
is not an arithmetic progression. Moreover, taking $H=\{0\}$ we have
$3|H|\le|2A|-|A|$; thus, $A$ is singular, a contradiction.
\end{proof}

\mysection{The case where $A$ meets at most two cosets}\label{s:twocoset}

The goal of this section is to prove the following result.
\begin{lemma}\label{l:twocoset}
Suppose that Theorem~\reft{aux} is wrong, and that $A\seq\Z_n$ is a
counterexample with $n$ smallest possible. Then $A$ meets at least three
cosets of any subgroup $F<\Z_n$ of index $|\Z_n/F|\ge 3$.
\end{lemma}

\begin{proof}
Suppose for a contradiction that $A$ meets at most two $F$-cosets. Since $A$
is not contained in a proper coset, this means that, in fact, $A$ meets
\emph{exactly} two $F$-cosets; say, $A=A_1\cup A_2$ with $A_i\seq g_i+F\
(i\in\{1,2\})$ and $g_1\not\equiv g_2\pmod F$. Notice that $\phi_F(g_2-g_1)$
generates $\Z_n/F$ as otherwise $A$ would be contained in a proper coset;
consequently, $2A$ meets exactly three $F$-cosets and
  $$ |2A| = |2A_1|+|A_1+A_2|+|2A_2|
                                   = |A+A_2| + |2A_1| = |A+A_1| + |2A_2|; $$
moreover, $2A_1$, $A_1+A_2$, and $2A_2$ reside in pairwise distinct
$F$-cosets.

Without loss of generality, we assume $|A_1|\ge|A_2|$.

\begin{claim}\label{m:2003a}
$A_1$ is a VSDS.
\end{claim}

\begin{proof}
Suppose first that $|A_2|\ge 2$. In this case $|A+A_2|\ge|A|+|A_2|$ by
Lemma~\refl{kemp1}, and we conclude that
  $$ |2A_1| = |2A| - |A+A_2|
       \le |2A| - |A| - |A_2| = |2A| - 2|A| + |A_1|. $$
Consequently,
  $$ |2A_1| < \frac14\,|A| + |A_1| \le \frac32\,|A_1|. $$

Now suppose that $|A_2|=1$ and, for a contradiction, that
$|2A_1|\ge\frac32\,|A_1|$. We have in this case $|A_1|\ge 3$ by
Lemma~\refl{numbers}, and also
\begin{equation}\label{e:0108b}
  \frac94\,|A| > |2A| = |A+A_2| + |2A_1| = |A| + |2A_1|
\end{equation}
implying
\begin{equation}\label{e:0108a}
  \frac32\,|A_1| \le |2A_1| < \frac54\,|A| = \frac54\,|A_1| + \frac54.
\end{equation}
As a result, $|A_1|\le 4$. In fact, we cannot have $|A_1|=3$ as
$|2A_1|\ge\frac32\,|A_1|$ would then imply $|2A_1|\ge 5$, whence $\frac
54\,|A|>|2A_1|\ge5$ leading to $|A|\ge 5>|A_1|+|A_2|$.

Thus, $|A_1|=4$ and then $|2A_1|=6=2|A_1|-2$ by~\refe{0108a}. Let
$H:=\pi(2A_1)$, and $k:=|A_1+H|/|H|$. By Kneser's theorem, $H$ is nonzero and
$6=|2A_1|=(2k-1)|H|$. It follows that either $k=1$ and $|H|=6$, or $k=2$ and
$|H|=2$. In the former case $A$ is contained in a union of two $H$-cosets
and, by~\refe{0108b},
  $$ |2A|-|A| = |2A_1| = 6 = |H|; $$
therefore, $A$ is $2$-regular. In the latter case $A_1$ is a union of two
$H$-cosets; therefore $A$ is contained in a union of three $H$-cosets and,
by~\refe{0108b},
  $$ |2A| - |A| = |2A_1| = 6 = 3|H|, $$
showing that $A$ is either $3$-regular, or singular.
\end{proof}

\smallskip
We therefore have $|2A_1|<\frac32|A_1|$; consequently, by Lemma~\refl{1.5},
the set $A_1$ is contained in a coset of a subgroup $L<\Z_n$ with
$|A_1|>\frac23|L|$ and $L=A_1-A_1$. Since $A_1$ is contained in an $F$-coset,
we have $L\le F$; consequently, $A_1$ and $A_2$ reside in distinct $L$-cosets
and moreover, the $L$-cosets of $2A_1$, $A_1+A_2$, and $2A_2$ are pairwise
distinct.

Write $A_2=B_1\longu B_k$ where the sets $B_i$ are nonempty, each of them is
contained in an $L$-coset, and the $k$ cosets are pairwise distinct. Since
$|A_1+A_2|=|A_1+B_1|\longp|A_1+B_k|\ge k|A_1|$, we have
  $$ \frac94\,|A| > |2A| = |2A_1|+|A_1+A_2|+|2A_2|
       \ge (k+1)|A_1| + |A_2| \ge \Big(\frac12k+1\Big)|A|, $$
whence $k\le 2$.

If $k=1$ then $A=A_1\cup B_1$. By Lemma~\refl{A'+A''},
  $$ |2A| = |2A_1|+|A_1+B_1|+|2B_1| \ge |L|+(|A|-1) + |B_1|, $$
implying $|2A|-|A|\ge |L|$; therefore $A$ is $2$-regular.

Thus, $k=2$. Without loss of generality, we assume that $|B_1|\ge|B_2|$.

As remarked above, the sets $2A_1$, $A_1+A_2=(A_1+B_1)\cup(A_1+B_2)$, and
$2A_2=2B_1\cup(B_1+B_2)\cup 2B_2$ reside in pairwise distinct $L$-cosets. It
is also immediately seen that the coset of $A_1+B_1$ is distinct from that of
$A_1+B_2$, and that the coset of $B_1+B_2$ is distinct from both the coset of
$2B_1$ and that of $2B_2$. Consequently, in the decomposition
\begin{equation}\label{e:1503a}
  2A = 2A_1 \cup(A_1+B_1)\cup(A_1+B_2)\cup2B_1\cup(B_1+B_2)\cup2B_2
\end{equation}
all six sets in the right-hand side reside in pairwise distinct $L$-cosets,
with the possible exception of the sets $2B_1$ and $2B_2$.

If at least one of $A_1$ and $B_1$ is not a coset of a subgroup of $\Z_n$,
then $|2A_1|+|2B_1|\ge|A_1|+|B_1|+1$; therefore, in view of the disjointness
and by Lemma~\refl{A'+A''},
\begin{align}
  |2A| &\ge |2A_1| + |2B_1|+ |A_1+B_1|+|B_2+(A_1\cup B_1)| \notag \\
       &\ge (|A_1| + |B_1| + 1) + |A_1| + (|A|-1) \notag \\
       &\ge \frac32\,|A_1| + \frac12\,(|A_1|+|B_1|+|B_2|)+|A| \label{e:FS} \\
       &=   \frac32\,|A_1|+\frac32\,|A| \notag \\
       &\ge \frac94|A|, \notag
\end{align}
a contradiction.

Thus, both $A_1$ and $B_1$ are cosets. Moreover, recalling that $A_1$ is
contained in an $L$-coset and $|A_1|\ge\frac23|L|$, we conclude that $A_1$ is
an $L$-coset. Let $K\le L$ be the subgroup such that $B_1$ is a $K$-coset.

If $K\ne\{0\}$, then we notice that the first five sets in the right-hand
side of~\refe{1503a} are $K$-periodic, and since $2A$ is aperiodic by
Lemma~\refl{2Aper}, the set $2B_2$ is not contained in the union of these
five sets. Therefore, as a slight modification of~\refe{FS},
\begin{align*}
  |2A| &\ge |2A_1| + |2B_1|+ |A_1+B_1|+|B_2+(A_1\cup B_1)| + 1 \\
       &\ge (|A_1| + |B_1|) + |A_1| + (|A|-1) + 1  \\
       &\ge \frac32\,|A_1| + \frac12\,(|A_1|+|B_1|+|B_2|) + |A| \\
       &\ge \frac94|A|,
\end{align*}
a contradiction.

We conclude that $A_1$ is an $L$-coset and $|B_1|=1$, as a result of which
also $|B_2|=1$.

If $2B_1\ne 2B_2$ then $|2(B_1\cup B_2)|=3$ and in view of
Lemma~\refl{numbers} we get
\begin{align*}
  |2A| &=   |2A_1| + |A_1+(B_1\cup B_2)| + |2(B_1\cup B_2)| \\
       &= 3|L| + 3 \\
       &=   3|A| - 3 \\
       &\ge \frac94|A|,
\end{align*}
a contradiction.

Therefore $2B_1=2B_2$ and $|2A|=3|L|+2=|A|+2|L|$.

Since $B_1$ and $B_2$ are in distinct $L$-cosets, from $2B_1=2B_2$ we
conclude that $|L|>2$.

If $|L|=3$ then $A$ is a union of an $L$-coset and a coset of the two-element
subgroup. As a result, $A$ is contained in a union of two cosets of the
six-element subgroup $H$ lying above $L$, while $|2A|-|A|=2|L|=|H|$; thus,
$A$ is $2$-regular.

Finally, if $|L|\ge 4$, then $|2A|=3|L|+2\ge\frac94(|L|+2)=\frac94\,|A|$, a
contradiction.
\end{proof}

\mysection{The case where $A$ meets exactly three cosets}
  \label{s:threecosets}

In this section we prove the following result.
\begin{lemma}\label{l:threecosets}
Suppose that Theorem~\reft{aux} is wrong, and that $A\seq\Z_n$ is a
counterexample with $n$ smallest possible. If $L<\Z_n$ is a subgroup such
that $\phi_L(A)$ is rectifiable, then $|\phi_L(A)|\ge 4$; that is, $A$ meets
at least four $L$-cosets.
\end{lemma}

As mentioned in the Introduction, the proof is rather technical and some
readers may prefer to skip it and proceed to the next section.
\begin{proof}

Aiming at a contradiction, we assume that $|\phi_L(A)|\le3$ and then, indeed,
$|\phi_L(A)|=3$ by Lemma~\refl{twocoset}. Let $A=A_1\cup A_2\cup A_3$ be the
$L$-coset decomposition of $A$; thus $2A$ is the union of the sets
  $$ A_1+A_2,\, A_2+A_3,\, A_3+A_1,\, 2A_1,\, 2A_2,\,2A_3. $$
Since $\phi_L(A)$ is rectifiable, by Lemma~\refl{3cyclic}, these sets
determine six pairwise distinct $L$-cosets except that, after a suitable
renumbering, the cosets determined by $2A_2$ and $A_1+A_3$ may coincide.

Suppose first that all the six sets listed are pairwise disjoint. By
Lemma~\refl{kemp1}, for each $i\in[1,3]$ we have
  $$ |A|+|A_i| \le |A+A_i|  = |A_1+A_i| + |A_2+A_i| + |A_3+A_i| $$
except if $|A_i|=1$ in which case the left-hand side must be replaced with
$|A|+|A_i|-1$. Since $|A|\ge 4$ in view of Lemma~\refl{numbers}, there is at
least one index $i$ with $|A_i|>1$. Therefore, taking the sum over all
$i\in[1,3]$ we obtain
  $$ 4|A| - 2 \le 2|2A| - (|2A_1|+|2A_2|+|2A_3|) \le 2|2A| - |A| . $$
Thus $|2A|\ge\frac52\,|A|-1$ and, consequently,
$\frac94\,|A|>\frac52\,|A|-1$; as a result, $|A|\le 3$, contradicting
Lemma~\refl{numbers}.

We therefore assume for the rest of the proof that $A_1+A_3$ is \emph{not}
disjoint from $2A_2$; hence, $2A$ meets exactly five $L$-cosets. Notice that
in this case, for any subgroup $H$ such that each of $A_1$, $A_2$, and $A_3$
is contained in an $H$-coset, the three cosets are in an arithmetic
progression.

We have
  $$ |2A| = |A_1+A_2| + |A_2+A_3|+|2A_1| + |2A_3| + |(A_1+A_3)\cup(2A_2)|; $$
our goal is to show that either
\begin{equation*}
  |2A| \ge \frac94\,|A|,
\end{equation*}
or there is a subgroup $H$ such that each of $A_1,A_2,A_3$ is contained in an
$H$-coset, and
\begin{equation*}
  |2A| \ge |A| + 2|H|
\end{equation*}
(in which case $A$ is $3$-regular). Once any of these estimates gets
established, we have reached a contradiction and the proof is over. We thus
assume that the estimates in question do not hold. We also make the following
assumptions:
\begin{itemize}
\item[i)]   $|A|\ge 4$ (by Lemma~\refl{numbers});
\item[ii)]  $|A+A_i|\ge|A|+|A_i|-1$ for any $i\in\{1,2,3\}$; moreover, if
    $|A_i|>1$, then the term $-1$ in the right-hand side can be dropped (by
    Lemma~\refl{kemp1});
\item[iii)] $|A_i+A_j|+|A_j+A_k|\ge |A|-1$ for any permutation $(i,j,k)$ of
    the index set $\{1,2,3\}$ (by Lemma~\refl{A'+A''} and in view of
    $(A_i+A_j)\cup(A_j+A_k)=A_j+(A_i\cup A_k)$).
\end{itemize}
These assumptions will be used throughout the proof without any further
explanations or references.

\begin{claim}\label{m:5/4}
We have
  $$ |2A_1| + |2A_2| + |2A_3| < \frac54\,|A| + 1. $$
Consequently, at least one of $A_1$, $A_2$, and $A_3$ is a VSDS.
\end{claim}

\begin{proof}
The first assertion follows from
\begin{multline*}
   \frac94\,|A| > |2A|
        \ge (|A_1+A_2| + |A_2+A_3|) + (|2A_1| + |2A_2| + |2A_3|) \\
          \ge |A|-1 + (|2A_1| + |2A_2| + |2A_3|),
\end{multline*}
the second is an immediate corollary of the definition of a VSDS and
Lemma~\refl{numbers}.
\end{proof}

\begin{claim}\label{m:singletons}
Among the sets $A_1$, $A_2$, and $A_3$, at most one is a singleton; thus,
$|A|\ge 5$.
\end{claim}

\begin{proof}
Suppose first that $|A_1|=|A_2|=1$. Then $|A|=|A_3|+2$ and if $A_3$ is not a
coset, then
\begin{multline*}
  |2A| \ge |A_1+A_3| + |A_2+A_3| + |2A_3| + |2A_1| + |A_1+A_2| \\
       =   2|A_3| + |2A_3| + 2
       \ge 3|A_3| + 3
       = 3|A|-3
       \ge \frac94\,|A|,
\end{multline*}
as wanted. If, on the other hand, $A_3$ is a coset, then arguing the same way
we get $|2A|\ge 3|A|-4$; that is, $|2A|-|A|\ge 2|A|-4=2|A_3|$ showing that
$A$ is $3$-regular.

Similarly, if $|A_1|=|A_3|=1$, then $|A|=|A_2|+2$ and either
\begin{multline*}
  |2A|\ge |A_1+A_2| + |2A_2| + |A_2+A_3| + |2A_1| + |2A_3| \\
                   = 2|A_2|+|2A_2|+2 \ge 3|A_2| + 3 = 3|A|-3 \ge \frac94\,|A|,
\end{multline*}
or $A_2$ is a coset, $|2A|\ge 3|A|-4$, and then $A$ is $3$-regular in view of
 $|2A|-|A| \ge 2|A|-4=2|A_2|$.
\end{proof}

\begin{claim}\label{m:3006}
If $A_2$ is not a VSDS, then both $A_1$ and $A_3$ are VSDS.
\end{claim}

\begin{proof}
Recalling Claim~\refm{5/4}, suppose for a contradiction that, say, $A_3$ is
the only VSDS among $A_1,A_2,A_3$; thus, $|2A_1|\ge \frac32\,|A_1|$ and
$|2A_2|\ge \frac32\,|A_2|$; furthermore, there is a finite subgroup $H$ such
that $A_3$ is contained in an $H$-coset and $|A_3|>\frac23|H|$. As a result,
\begin{align}
  |2A| &\ge (|A_1+A_2|+|A_2+A_3|) + |2A_1|+|2A_2|+|2A_3| \notag \\
    &\ge |A|-1 + \frac32|A_1| + \frac32\,|A_2| + |H| \notag \\
    &=   \frac52\,|A|-\frac32|A_3|+|H|-1 \notag \\
    &\ge \frac52\,|A|-\frac12\,|H|-1. \label{e:3005b}
\end{align}

On the other hand, if $A_2$ is not contained in an $H$-coset, then
$|A_2+A_3|\ge2|A_3|$ resulting in
\begin{align}
  |2A| &\ge |A_1+A_2|+|A_2+A_3| + |2A_1| + |2A_2| + |2A_3| \notag \\
       &\ge  \frac12(|A_1|+|A_2|) + 2|A_3|
                    + \frac32\,|A_1| + \frac32\,|A_2| + |H| \notag \\
       &= 2|A|+|H|. \label{e:3005d}
\end{align}

Multiplying~\refe{3005b} by $2$ and taking the sum with~\refe{3005d} we get
  $$ 3|2A| \ge 7|A| - 2, $$
whence
  $$ \lcl\frac94\,|A|\rcl-1 \ge |2A| \ge \frac73\,|A|-\frac23. $$
This is possible only for $|A|=5$. Recalling that $A_3$ is a VSDS while $A_1$
and $A_2$ are not, we conclude that in this case $|A_1|=|A_2|=2$ and
$|A_3|=1$. This further results in $|2A_1|=|2A_2|=3$ and $|A_1+A_2|\ge 3$
(for the last estimate notice that $|A_1+A_2|=2$ would mean that $A_1$ is
contained in the period of $A_2$ and vice versa, meaning that $A_1=A_2$ is
the two-element subgroup, while $A_1$ and $A_2$ are in fact disjoint).
Consequently,
\begin{multline*}
  |2A| \ge |A_1+A_2|+|A_2+A_3|+|2A_1|+|2A_2|+|2A_3| \\
       \ge     3    +    2    +  3   +  3  +   1  = 12 > \frac94\,|A|,
\end{multline*}
a contradiction showing that $A_2$ is contained in an $H$-coset.

We now show that $A_1$ is contained in an $H$-coset, too. Assuming it is not,
we have
  $$ |A_1+A_2| \ge \max\{|A_1|,2|A_2|\} \ge \frac38\,|A_1| + \frac54\,|A_2| $$
and, similarly,
  $$ |A_3+A_1| \ge \max\{|A_1|,2|A_3|\}
                                   \ge \frac38\,|A_1| + \frac54\,|A_3|. $$
Furthermore, $|2A_1|\ge\frac32\,|A_1|$ (as we assume that $A_1$ is not a
VSDS), and trivially, $|2A_3|\ge|A_3|$ and $|A_2+A_3|\ge|A_2|$. Therefore,
\begin{align*}
  \frac94\,|A| &> |A_1+A_2|+|A_3+A_1|+|A_2+A_3|+|2A_1|+|2A_3| \\
     &\ge \Big(\frac34\,|A_1| + \frac54\,|A_2| + \frac54\,|A_3|\Big)
            + |A_2| + \frac32\,|A_1| + |A_3| \\
     &= \frac94\,(|A_1|+|A_2|+|A_3|),
\end{align*}
a contradiction.

We have thus shown that each of $A_1,A_2$, and $A_3$ is contained in an
$H$-coset. Furthermore, $|A_2|\le\frac23|H|<|A_3|$; hence, by
Lemma~\refl{Olson}, either $|A_2+A_3|\ge|A_2|+\frac12|A_3|$, or $A_3$ is
contained in a coset of the period $\pi(A_2+A_3)$. In the latter case we have
$H=A_3-A_3\seq\pi(A_2+A_3)$; since, on the other hand, $A_2+A_3$ is contained
in an $H$-coset, we actually have $|A_2+A_3|=|H|$. Therefore,
\begin{align*}
  |2A| &\ge (|A_1+A_2|+|A_3+A_1|) + |A_2+A_3| + |2A_1| + |2A_3| \\
    &\ge (|A|-1) + 2|H| + |2A_1| \\
    &\ge |A|+2|H|
\end{align*}
so that $A$ is $3$-regular.

Assuming thus that $|A_2+A_3|\ge|A_2|+\frac12|A_3|$, in view of
  $$ |2A_3| = |H| \ge \max\Big\{|A_3|,\frac32\,|A_2| \Big\}
                                     > \frac34\,|A_3| + \frac14\,|A_2| $$
we get the desired
\begin{align*}
  |2A| &\ge (|A_1+A_2| + |A_3+A_1|) + |A_2+A_3| + |2A_1| + |2A_3| \\
    &\ge |A|-1 + \Big(|A_2| + \frac12|A_3|\Big) + \frac32\,|A_1|
                              + \Big(\frac34\,|A_3| + \frac14\,|A_2|\Big) \\
    &= |A|-1 + \frac54\,|A| + \frac14\,|A_1| \\
    &\ge \frac94\,|A|.
\end{align*}
\end{proof}

We now consider two cases, according to whether $A_2$ is or is not a VSDS.

\newcases
\case{$A_2$ is a VSDS}

Suppose that $A_2$ is a VSDS, and let $H:=A_2-A_2$.

\begin{claim}\label{m:3H}
We have $|A_1+H|+|A_3+H|\ge 3|H|$.
\end{claim}

\begin{proof}
Suppose for a contradiction that each of $A_1$ and $A_3$ is contained in a
single $H$-coset. Since $|2A_2|=|H|$, using the trivial estimates
$|2A_i|\ge|A_i|$ and $|A_2+A_i|\ge|A_2|$, where $i\in\{1,3\}$, we get
\begin{equation}\label{e:0407a}
  \frac94\,|A| > |2A| = |2A_1|+|2A_3|+|A_1+A_2|+|A_2+A_3|+|2A_2|
       \ge |A| + |A_2| + |H|
\end{equation}
and we conclude that
\begin{equation}\label{e:2506a}
  \frac54\,|A| > |A_2|+|H|.
\end{equation}
If $|A_1|+|A_2|\le |H|$ and $|A_3|+|A_2|\le |H|$, then taking the sum we get
\begin{equation}\label{e:2606a}
  2|H| \ge |A|+|A_2|.
\end{equation}
Combining~\refe{2506a} and~\refe{2606a},
  $$ |A_2| < \frac54\,|A|-|H| \le \frac32\,|H| - \frac54\,|A_2| $$
whence $|A_2|<\frac23\,|H|$, a contradiction showing that either
$|A_1|+|A_2|>|H|$, or $|A_3|+|A_2|>|H|$ holds true. Assuming the latter for
definiteness, by the pigeonhole principle we have $|A_2+A_3|=|H|$, and then
from~\refe{0407a} we obtain $|2A|\ge |A|+2|H|$; hence, $A$ is $3$-regular.
\end{proof}

\begin{claim}\label{m:2206d}
We have $|A_2|<\frac14\,|A|$.
\end{claim}

\begin{proof}
Assuming that, say, $A_1$ meets at least two $H$-cosets
(cf.~Claim~\refm{3H}), we have $|A_1+A_2|\ge2|A_2|$ and then
\begin{multline*}
  \frac94\,|A|
    > |2A|
    \ge |A_3+A| + |2A_1| + |A_1+A_2| \\
    \ge (|A| + |A_3| - 1) +|A_1| + 2|A_2|
    =   2|A| + |A_2| -1.
\end{multline*}
To complete the proof, we show that the term $-1$ in the right-hand side can
be dropped. It is easy to see that otherwise the following conditions are
meat simultaneously: $|A_3|=1$, there is a subgroup $K$ such that $A_1$ is a
$K$-coset, $|A_1+A_2|=2|A_2|$, and $2A_2\seq A_1+A_3$. The first and the last
conditions show that $A_1$ contains an $H$-coset; hence, $K\ge H$. Therefore
$A_1+A_2$ is a $K$-coset, and the condition $|A_1+A_2|=2|A_2|$ shows that
$|K|=2|H|$ and that $A_2$ is an $H$-coset. Therefore $|A|=|K|+|H|+1$,
$|A_2|=|H|$, and
  $$ |2A| = |A_3+A|+|A_2+A_1|+|2A_1| = |A| + 2|K|; $$
therefore $A$ is $3$-regular.
\end{proof}

To complete the treatment of the present case where $A_2$ is a VSDS, we prove
the following claim which is in a clear contradiction with the previous one.
\begin{claim}\label{m:2206e}
We have $|A_2|\ge\frac14\,|A|$.
\end{claim}

\begin{proof}
Let $\del:=|2A_2\stm(A_1+A_3)|$ and
  $$ \del_i :=
        \begin{cases}
           |2A_i|-|A_i| &\text{if}\ |A_i|>1 \\
           -1 &\text{if}\ |A_i|=1
        \end{cases} , \qquad i\in\{1,3\}. $$
The quantity $\del_i$ shows whether $A_i$ is a singleton ($\del_i=-1$), a
coset of a nonzero subgroup ($\del_i=0$), or neither ($\del_i>0$).

By Lemma~\refl{kemp1}, we have $|A+A_i|+|2A_i|\ge |A|+2|A_i|+\del_i$,
$i\in\{1,3\}$. Consequently, taking the sum of
\begin{align*}
  |2A| &\ge |A_1+A| + |A_3+A| - |A_1+A_3| + \del
  \intertext{and}
  |2A| &\ge |A_2+(A_1\cup A_3)| + |A_3+A_1| + |2A_1| + |2A_3| + \del
\end{align*}
we get
\begin{align*}
  \frac92\,|A| - \frac12
       &\ge 2|2A| \\
       &\ge (|A_1+A|+|2A_1|) + (|A_3+A|+|2A_3|)
                                          + |A_2+(A_1\cup A_3)| + 2\del \\
       &\ge 2|A| + 2|A_1| + 2|A_3| + (|A|-1) + \del_1 + \del_3 + 2\del \\
       &=   5|A| - 2|A_2|  + \del_1 + \del_3 + 2\del - 1
\end{align*}
whence
  $$ |A_2| \ge \frac14\,|A|
                        + \frac12\,(\del_1+\del_3) + \del - \frac14. $$
With Claim~\refm{singletons} in mind, we thus assume for the rest of the
proof that $\del_1+\del_3\in\{-1,0\}$, that $\del=0$ (that is,
 $2A_2\seq A_1+A_3$), and (switching $A_1$ and $A_3$, if needed) that
$\del_1\le \del_3$; that is, either $\del_1=-1$ and $\del_3\in\{0,1\}$, or
$\del_1=\del_3=0$. Moreover, by Claim~\refm{3H}, in each of these cases we
can assume that $A_3$ meets at least two $H$-cosets. (If $A_3$ meets just one
$H$-coset, then $A_1$ meets at least two; hence $\del_1\ge 0$, leading to
$\del_1=\del_3=0$, and we switch $A_1$ and $A_3$ without violating any of the
assumptions.)

Suppose first that $\del_1=-1$ and $\del_3=0$; thus, $|A_1|=1$ and $A_3$ is a
coset of a nonzero subgroup, say $K$. Since $2A_2\seq A_1+A_3$, and since
$2A_2$ is an $H$-coset, while $A_1+A_3$ is a $K$-coset, we have $H\le K$. A
simple counting shows now that $|A|=|A_2|+|K|+1$ while $|2A|=3|K|+|A_2|+1$;
therefore, $|2A|-|A|=2|K|$ and $A$ is $3$-regular.

Next, we consider the case where $\del_1=-1$ and $\del_3=1$; that is, $A_1$
is a singleton, and $A_3$ is not a coset. By Claim~\refm{singletons}, we have
$|H|\ge|A_2|\ge 2$. Furthermore, in view of $2A_2\seq A_1+A_3$, the set $A_3$
contains an $H$-coset; moreover, the containment is proper since $A_3$ meets
at least two $H$-cosets. As a result,
  $$ |A_2+A_3| \ge \max\{|A_2|+1,|A_3|\}
                             \ge \frac12\,(|A_2|+1+|A_3|) = \frac12\,|A| $$
and, consequently,
\begin{multline*}
  \frac94\,|A| > |2A|
  = |A_1+A| + |A_2+A_3| + |2A_3| \\
  \ge |A| + \frac12\,|A| + (|A_3| + 1)
                                      = \frac52\,|A| - |A_2|
\end{multline*}
which gives the desired estimate $|A_2|\ge\frac14\,|A|$.

Finally, we consider the case where $\del_1=\del_3=0$; that is, $A_1$ is a
coset of a nonzero subgroup $H_1$, and  $A_3$ is a coset of a nonzero
subgroup $H_3$. Since $2A$ is aperiodic, and $2A_2\seq A_1+A_3$, we have
$H_1\cap H_3=\{0\}$. Furthermore, $|A|=|H_1|+|A_2|+|H_3|$ and
%        $$       11\    33\      13\         12\    23 $$
\begin{align*}
  |2A| &=   |2A_1|+|2A_3| + |A_1+A_3| + |A_1+A_2|+|A_2+A_3| \\
       &\ge |H_1|+|H_3| + |H_1||H_3| + |H_1|+|H_3| \\
       &=   (|H_1|-2)(|H_3|-2) + 4|H_1| + 4|H_3| - 4 \\
       &\ge 4|A| - 4|A_2| - 4.
\end{align*}
If we had $|A_2|\le\frac14\,|A|-\frac14$, this would further lead to
  $$ \frac94\,|A| > |2A| \ge 3|A|-3 $$
contradicting the assumption $|A|\ge 4$.
\end{proof}

\case{$A_2$ is not a VSDS}

Recall that, by Claim~\refm{3006}, in this case both $A_1$ and $A_3$ are
VSDS. Assuming for definiteness that $|A_3|\ge|A_1|$, consider the subgroup
$H:=A_3-A_3$.

\begin{claim}
$A_2$ is contained in a single $H$-coset.
\end{claim}

\begin{proof}
Assuming the opposite, we have $|A_2+A_3|\ge2|A_3|$ and, by
Corollary~\refc{AVSDS},
  $$ |A_1+A_2|\ge\max\big\{|A_1|,|A_2|,\min\{2|A_1|,\frac32\,|A_2|\}\big\}. $$
Consequently,
\begin{align*}
   \frac94\,|A|
    &> |2A| \\
    &\ge |2A_1| + |2A_3| + |A_1+A_2| + |A_2+A_3| + |2A_2| \\
    &\ge |A_1|+|A_3| + \max\{|A_1|,|A_2|,\min\{2|A_1|,\frac32\,|A_2|\}\}
                                                    +2|A_3| + \frac32\,|A_2|
\end{align*}
leading to
  $$ \max\{|A_1|,|A_2|,\min\{2|A_1|,\frac32\,|A_2|\}\}
      < \frac54\,|A_1| + \frac34\,|A_2| - \frac34\,|A_3|
         \le \frac12\,|A_1| + \frac34\,|A_2|. $$
However, the resulting estimate is easily shown to be wrong by analyzing the
four cases where $|A_1|\le\frac12\,|A_2|$,
$\frac12\,|A_2|\le|A_1|\le\frac34\,|A_2|$,
$\frac34\,|A_2|\le|A_1|\le\frac32\,|A_2|$, and $|A_1|\ge\frac32\,|A_2|$.
(Less rigorous, but more convincing is to let $t:=|A_1|/|A_2|$, rewrite the
inequality in question as
$\max\{1,t,\min\{2t,\frac32\}\}<\frac12\,t+\frac34$, and plot both sides, as
functions of $t$).
\end{proof}

Next, we show that the set $A_1$ is contained in a single $H$-coset, too.

\begin{claim}
$A_1$ is contained in a single $H$-coset.
\end{claim}

\begin{proof}
Assuming the opposite, the sum $A_1+A_3$ meets at least two $H$-cosets, and
has at least $|A_3|$ elements in every $H$-coset that it meets. Consequently,
$|(2A_2)\cup(A_1+A_3)|\ge |2A_2|+|A_3|\ge \frac32\,|A_2|+|A_3|$. Therefore
\begin{align*}
  \frac94\,|A| &> |2A| \\
    &\ge (|A_1+A_2| + |A_2+A_3|) + |2A_1| + |2A_3|
                                                  + |(2A_2)\cup(A_1+A_3)| \\
    &\ge (|A|-1) + |A_1| + |A_3| + \left(\frac32\,|A_2|+|A_3|\right) \\
    &\ge \frac52\,|A|-1
\end{align*}
contradicting Lemma~\refl{numbers}.
\end{proof}

We have thus shown that each of $A_1,A_2$, and $A_3$ is contained in an
$H$-coset. We also recall that, by our present assumptions, $A_1$ and $A_3$
are VSDS, while $A_2$ is not, and that $A_3-A_3=H$ and $|A_1|\le|A_3|$; as a
result, $|A_2|\le\frac23\,|H|\le|A_3|$.

\subcase{$\max\{|A_1|,|A_2|\}\ge\frac12|A_3|$} If $|A_2|\ge\frac12\,|A_3|$,
then in view of $|A_3|>\frac23\,|H|$ we have $|A_2|+|A_3|>|H|$. Therefore
$A_2+A_3$ is an $H$-coset and
\begin{align*}
  |2A| &\ge |A_1+A_2| + |A_2+A_3| + |A_3+A_1| + |2A_1| + |2A_3| \\
    &\ge |A_2| + |H| + |A_3| + |A_1| + |H| \\
    &= |A| + 2|H|
\end{align*}
so that $A$ is $3$-regular.

Similarly, if $|A_1|\ge\frac12\,|A_3|$, then $|A_1|+|A_3|>|H|$. Therefore
$A_1+A_3$ is an $H$-coset and then
\begin{align*}
  |2A| &\ge (|A_1+A_2| + |A_2+A_3|) + |A_3+A_1| + |2A_1| + |2A_3| \\
    &\ge (|A|-1) + |H| + 1 + |H| \\
    &= |A| + 2|H|
\end{align*}
shows that $A$ is $3$-regular.

\subcase{$\max\{|A_1|,|A_2|\}<\frac12|A_3|$} We have
\begin{align*}
  \frac94\,|A|-\frac14 &\ge |2A| \\
     &\ge (|A_1+A_2| + |A_2+A_3|) + |A_1+A_3| + |2A_1| + |2A_3| \\
     &\ge (|A|-1) + |A_3| + |A_1| + |A_3| \\
     &\ge |A_1| + \frac54\,|A_3|
         + \frac34\,\left(\frac13\,|A_1|+\frac53\,|A_2| + 1\right) + |A| - 1
        = \frac94\,|A| - \frac14.
\end{align*}
This shows that $|2A_1|=|A_1|$ and $|2A_3|=|A_3|$; that is, both $A_1$ and
$A_3$ are cosets. Since $A_3-A_3=H$ and $A_1$ is contained in an $H$-coset,
we conclude that $A_3$ is an $H$-coset and that there is a subgroup $K\le H$
such that $A_1$ is a $K$-coset. In this case $|A|=|K|+|A_2|+|H|$ and from
  $$ |2A_1|=|K|,\ |A+A_3|=3|H|,\ |A_2+A_1|\ge|A_2|, $$
we get $|2A|\ge 3|H|+|K|+|A_2|$; hence, $|2A|-|A|\ge 2|H|$.
\end{proof}

\mysection{Character sums and partial rectification}\label{s:charsum}

This section combines a character-sum argument and a combinatorial reasoning.
Its central component is a lemma which, loosely speaking, shows that over
90\% of a counterexample set must be well-structured. The lemma is a version
of \cite[Proposition~4.2]{b:df} incorporating a critically important trick
from \refb{ls}. Historically, quoting from~\refb{df}, ``the underlying idea
$\<$of the lemma$\>$ comes from~\refb{f1}\, (\ldots)\! where the case of
prime modulus $n$ was dealt with''.

Recall that an arithmetic progression in an abelian group is \emph{primitive}
if its difference generates the group.
\begin{lemma}\label{l:DF}
Suppose that Theorem~\reft{aux} is wrong. If $A\seq\Z_n$ is a counterexample
with $n$ smallest possible, then there exist a subgroup $H<\Z_n$ of index
$m:=n/|H|\ge 37$, and a primitive arithmetic progression $P\seq\Z_n$ with
$|P|\le(m+1)/2$, such that $|(P+H)\cap A|>0.9|A|$.
\end{lemma}

\begin{proof}
We assume that $|2A|<\min\{\frac94\,|A|,n\}$ (since $A$ satisfies the
assumptions of Theorem~\reft{aux}), that $|2A|-|A|\le C_0^{-1}n$ (since $A$
fails to satisfy the conclusion of the theorem), and that $|A+B|\ge|A|+|B|$
holds for any subset $B\seq\Z_n$ with $2\le|B|\le|A|$ (in view of
Lemma~\refl{kemp1}). Also, $|2A|\ge2|A|\ge 8$ and $n\ge4C_0$, see
Lemmas~\refl{numbers} and~\refl{kemp1}.

For a finite subset $B$ and an element $x$ of an abelian group, we let
$B^{(x)}:=B\cap(B+x)$; therefore, $|B^{(x)}|$ is the number of
representations of $x$ as a difference of two elements of $B$, and in
particular $|B^{(x)}|=0$ if $x\notin B-B$. We have
  $$ \sum_{x\in B-B} |B^{(x)}| = |B|^2 $$
and
\begin{equation}\label{e:kko}
  B^{(x)}+B \seq (2B)^{(x)};
\end{equation}
the latter relation, often called the \emph{Katz-Koester observation}, can be
proved as follows:
  $$ B^{(x)}+B = (B\cap(B+x))+B \seq (2B) \cap ((2B)+x) = (2B)^{(x)}. $$
We also have
  $$ \sum_{x\in B-B} |B^{(x)}|^2 = \E(B), $$
where $\E(B)$ (standardly called the \emph{energy} of $B$) is the number of
quadruples $(b_1\longc b_4)\in B^4$ with $b_1+b_2=b_3+b_4$. We recall the
basic estimate $\E(B)\ge|B|^4/|2B|$ following easily from the Cauchy-Schwartz
inequality.

Let $\tau:=|2A|/|A|$. Denoting by $\wh{A}$ the counting-measure Fourier
transform of the indicator function of the set $A$, and similarly for the
indicator function of the sumset $S:=2A$, we have
\begin{equation}\label{e:1404a}
  \frac1n\, \sum_{\chi\in\widehat{\Z_n}} |\hA(\chi)|^2 |\hS(\chi)|^2
     =  \sum_{x\in A-A} |A^{(x)}||S^{(x)}|
                                \ge  \sum_{x\in A-A} |A^{(x)}||A+A^{(x)}|;
\end{equation}
here the equality follows by a direct computation, both sums involved
counting the number of solutions to $a_1-a_2=s_1-s_2$ with $a_1,a_2\in A$ and
$s_1,s_2\in S$, and the inequality follows from~\refe{kko}. Let $D$ be the
set of all those $x\in\Z_n$ with $|A^{(x)}|=1$, and let $N:=|D|$. By
Lemma~\refl{kemp1} we have $|A+A^{(x)}|\ge|A|+|A^{(x)}|$ unless $x\in D$.
Consequently, denoting the sum in the left-hand side of~\refe{1404a} by
$\sig$,
\begin{align*}
  \sig &\ge \sum_{x\in A-A} |A^{(x)}||A+A^{(x)}| \\
       &\ge \sum_{\substack{x\in A-A \\ x\ne 0}} |A^{(x)}|(|A|+|A^{(x)}|)
                                      - \sum_{x\in D} |A^{(x)}| + |A||S| \\
       &\ge \sum_{x\in A-A} |A^{(x)}|(|A|+|A^{(x)}| - N + |A||S| - 2|A|^2 \\
       &=   |A|^3 + \E(A) - N - (2-\tau)|A|^2
\end{align*}
where the terms $|A||S|$ and $-2|A|^2$ arise from considering the summand
corresponding to $x=0$. We conclude that
\begin{equation}\label{e:1805a}
  \sig \ge |A|^3 + \frac{|A|^3}{\tau} + (\tau-2)|A|^2 - N.
\end{equation}
We split the sum in the left-hand side into two parts,
  $$ \sig_0 = \frac1n\, \sum_{\substack{\chi\in\hZn \\ |\ker\chi|\ge n/36}}
                                                |\hA(\chi)|^2 |\hS(\chi)|^2 $$
and
  $$ \sig_1 = \frac1n\, \sum_{\substack{\chi\in\hZn \\ |\ker\chi|<n/36}}
                      |\hA(\chi)|^2 |\hS(\chi)|^2 $$
(the constant $36$ was found by a hindsight optimization). Let $\phi$ denote
Euler's totient function. For any divisor $d\mid n$, there are exactly
$\phi(d)$ characters $\chi\in\hZn$ with $|\ker\chi|=n/d$. Therefore,
estimating trivially the first sum,
  $$ \sig_0 < \Phi(n)\,|A|^2|S|^2 = \Phi(n) \,\tau|A|^3, $$
where
  $$ \Phi(n) = \frac1n\,\sum_{\substack{1\le d\le 36 \\ d\mid n}} \phi(d). $$
Let $\eps:=\frac4{2025}$. If $n>200,\!475$, then
  $$ \Phi(n) < \frac1{200,\!475} \sum_{1\le d\le 36} \phi(d) < \eps, $$
and a computer verification shows that the resulting estimate $\Phi(n)<\eps$
also holds for all values $92,\!400<n\le 200,\!475$. Recalling that $n\ge
4C_0> 92,\!400$ by Lemma~\refl{numbers}, we therefore have
\begin{equation}\label{e:1805b}
  \sig_0 < \eps \tau |A|^3.
\end{equation}

For the second sum, letting
  $$ \eta := \max_{\substack{\chi\in\hZn \\ |\ker\chi|<n/36}}
                                                  |\hA(\chi)| / |A| $$
and using Parseval's identity, we get
  $$ \sig_1 \le \frac1n\,\eta^2|A|^2
       \sum_{\substack{\chi\in\hZn \\ |\ker\chi|<n/36}} |\hS(\chi)|^2
                                 \le \eta^2|A|^2 |S| = \eta^2\tau|A|^3. $$
Combining this estimate with~\refe{1805a} and~\refe{1805b}, we obtain
  $$ (\eta^2 + \eps) \tau|A|^3
                    > |A|^3 + \frac{|A|^3}{\tau} + (\tau-2)|A|^2 - N; $$
that is,
  $$ \eta^2 + \eps \ge \frac1\tau + \frac1{\tau^2}
                            + \frac{\tau-2}{\tau|A|} - \frac N{\tau|A|^3}. $$
Since $|A|\ge 4$, and using the trivial bound $N\le|A|^2$, the right-hand
side is easily verified to be a decreasing function of $\tau$; therefore, by
Lemma~\refl{Mantel},
  $$ \eta^2 + \eps > \frac49 + \frac{16}{81}
       + \frac4{9|A|^3}\left(\frac14\,|A|^2-N \right)
           \ge \frac49 + \frac{16}{81} = \frac{52}{81} $$
whence $\eta>0.8$.

Thus, there exists a character $\chi\in\hZn$ such that $|\ker\chi|<n/36$ and
  $$ |\hA(\chi)| > 0.8|A|. $$
Letting $H:=\ker\chi$ and $m:=n/|H|$ (so that $m\ge 37$, $H=m\Z_n$, and
$\Z_n/H\cong\Z_m$), there is a zero-kernel character
$\zet\in\widehat{\Z_n/H}$ such that $\chi=\zet\circ\phi_H$, where
$\phi_H\colon\Z_n\to\Z_n/H$ is the canonical homomorphism. In terms of this
character $\zet$, the last estimate can be rewritten as
  $$ \Big| \sum_{a\in A} \zet(\phi_H(a)) \Big| > 0.8|A|. $$
The summands in the left-hand side are complex roots of unity, and by
Lemma~\refl{Fpoints}, there exists a subset $A'\seq A$ of size
$|A'|>\frac12\,(1+0.8)|A|=0.9|A|$, and an open arc $\cC$ of the unit circle,
of angle measure $\pi$, such that $\zet(\phi_H(a))\in\cC$ for all $a\in A'$.
The arc $\cC$ contains at most $\lfloor(m+1)/2\rfloor$ roots of unity of
degree $m$, which are in a geometric progression. As a result, the set
$\phi_H(A')$ is contained in a primitive arithmetic progression $Q\seq\Z_n/H$
of size $|Q|\le (m+1)/2$; hence,
\begin{equation}\label{e:1401a}
  A' \seq \phi_H^{-1}(Q).
\end{equation}

Fix $c,d\in\Z_n$ such that $c+H$ and $d+H$ are the initial term and the
difference of the progression $Q$, respectively, and $d$ generates $\Z_n$;
the latter is possible since $d+H$ generates $\Z_n/H$. Letting
$P:=\{c,c+d\longc c+(|Q|-1)d\}\seq\Z_n$, we have $\phi_H(P)=Q$, whence
$\phi_H^{-1}(Q)=P+H$. This completes the proof in view of~\refe{1401a}.
\end{proof}

\mysection{Proof of Theorem~\reft{aux}}\label{s:final}

Suppose that the theorem is wrong. Let $n$ be the smallest positive integer
for which the assertion fails, and let $A\seq\Z_n$ be a counterexample set
satisfying the assumptions, but not the conclusion of the theorem. In
particular, $A$ is not contained in a proper coset, $n\ge4C_0$,
 $4\le |A|\le C_0^{-1}n$, and $8\le|2A|\le 2C_0^{-1}n$ by
Lemma~\refl{numbers}; also, $2A$ is aperiodic by Lemma~\refl{2Aper}.

Applying  Lemma~\refl{DF}, we find a subgroup $L<\Z_n$ of index
 $m:=n/|L|\ge 37$, and a primitive arithmetic progression $Q_0\seq\Z_n$ with
$|Q_0|\le(m+1)/2$ such that the set $A':=(Q_0+L)\cap A$ has size
$|A'|>0.9|A|$. The condition $|Q_0|\le(m+1)/2$ along with the primitivity of
$Q_0$ ensure that $\phi_L(Q_0)$ is rectifiable. Thus, $\phi_L(A')$ is
contained in a rectifiable subset of $\Z_n/L$; hence, is itself rectifiable.
Let $A'':=A\stm A'$. We observe that the $L$-cosets determined by $A'$ are
distinct from those determined by $A''$: $(A'+L)\cap(A''+L)=\est$. Also,
\begin{equation}\label{e:0308a}
  |2A'| \le |2A| < \frac94\,|A| < \frac52\,|A'|.
\end{equation}

It suffices to prove that $\phi_L(A)$ is rectifiable, as in this case
$|\phi_L(A)|\ge 4$ by Lemma~\refl{threecosets}, and applying
Proposition~\refp{combo} we conclude that $A$ is \emph{not} a counterexample.

\begin{claim}\label{m:A''ne}
The set $A''$ is nonempty.
\end{claim}

\begin{proof}
If $A''=\est$, then $A=A'$; as a result, $\phi_L(A)=\phi_L(A')$ is
rectifiable.
\end{proof}

In view of $|A''|<0.1|A|$, as an immediate corollary of Claim~\refm{A''ne} we
have
\begin{equation}\label{e:ac1}
  |A''|<\frac19\,|A'|\quad \text{and}\quad |A|\ge 11.
\end{equation}

\begin{claim}\label{m:A'per}
The set $A'$ is not contained in a proper coset.
\end{claim}

\begin{proof}
Suppose that $A'$ is contained in a proper coset, and let $g+F$, with
$g\in\Z_n$ and $F<\Z_n$, be the smallest coset containing $A'$.
 If $a_1\longc a_k$ list representatives of the $F$-cosets intersecting
$A''$, other than the coset $g+F$ (which can possibly contain elements of
$A''$) then $2A',a_1+A'\longc a_k+A'$ reside in pairwise distinct $F$-cosets
and, therefore, are disjoint. As a result
  $$ (k+1)|A'| \le |2A'|+|a_1+A'|\longp |a_k+A'|\le |2A|
                                          < \frac94\,|A| < \frac52|A'|, $$
showing that $k\le 1$. Indeed, $k=1$ as if we had $k=0$, then $A$ were
contained in $g+F$, which is a proper coset.

Reverting the last computation and taking the result a little further,
  $$ \frac52\,|A'| > \frac94|A| > |2A| \ge |2A'|+|a_1+A'| $$
whence $|2A'|<\frac32|A'|$. Therefore $A'$ is a VSDS; moreover, by
Lemma~\refl{1.5} and the minimality of $F$, we have $A'-A'=F$, $|2A'|=|F|$,
and $|A'|>\frac23|F|$. Now from $|F|<\frac32\,|A'|<\frac32|A|$ and
Lemma~\refl{numbers} we see that $|F|>\frac13\,n$. On the other hand,
$A\seq(g+F)\cup(a_1+F)$, contradicting Lemma~\refl{twocoset}.
\end{proof}

Recall that we have defined $m:=n/|L|$.
\begin{claim}\label{m:A'prog}
For any subgroup $K\le L$, the set $\phi_K(A')$ is not contained in an
arithmetic progression with $\lcl\frac m6\rcl$ or fewer terms.
\end{claim}

\begin{proof}
If, for some $a,d\in\Z_n$ and $k\ge 1$ we have
  $$ \phi_K(A')\seq \{\phi_K(a)+i\phi_K(d)\colon i\in[0,k-1] \}, $$
then
  $$ A'\seq \bigcup_{i\in[0,k-1]} (a+id+K) , $$
whence
  $$ \phi_L(A') \seq \{ \phi_L(a)+i\phi_L(d)\colon i\in[0,k-1]\}. $$
Therefore, it suffices to prove the assertion in the special case where
$K=L$.

By Claim~\refm{A'per}, the set $A'$ is not a VSDS; therefore
\begin{equation}\label{e:0304a}
  |2A'|\ge\frac32|A'|
\end{equation}
by Lemma~\refl{1.5}.

If $A$ contained an element $a\notin 2A'-A'$, then $a+A'$ would be disjoint
from $2A'$, and from~\refe{0304a} we would get
  $$ |2A| \ge |a+A'|+|2A'| \ge \frac52\,|A'| > \frac94\,|A| $$
(cf.~\refe{0308a}). Thus,
\begin{equation}\label{e:2A'A'}
  A \seq 2A'-A'.
\end{equation}

Suppose now that $\phi_L(A')$ is contained in an arithmetic progression with
$k\le\lcl\frac m6\rcl$ terms. Then, by~\refe{2A'A'}, the set $\phi_L(A)$ is
contained in a progression with $3k-2\le\frac{m+1}2$ terms. By
Claim~\refm{A'per}, the difference of this progression generates $\Z_n/L$. It
follows that $\phi_L(A)$ is rectifiable.
\end{proof}

By Lemma~\refl{threecosets}, if $\phi_L(A)$ is rectifiable, then
$|\phi_L(A)|\ge 4$. We now show that the conclusion $|\phi_L(A)|\ge 4$ holds
true regardless of the rectifiability of $\phi_L(A)$.

\begin{claim}
The set $A$ determines at least four distinct $L$-cosets; that is,
$|\phi_L(A)|\ge 4$.
\end{claim}

\begin{proof}
With Lemma~\refl{twocoset} in mind, suppose for a contradiction that $A$
determines exactly three $L$-cosets. By Claims~\refm{A''ne} and~\refm{A'per},
the set $A'$ meets exactly two of these three cosets. Hence,
$|\phi_L(A')|=2$; therefore, $\phi_L(A')$ is a (two-term) progression,
contradicting Claim~\refm{A'prog}.
\end{proof}

Write $s:=|\phi_L(A')|$, and let $A'=A_1\longu A_s$ where each of the sets
$A_1\longc A_s$ is contained in an $L$-coset, the cosets are pairwise
disjoint, and $|A_1|\longge|A_s|>0$. By Claims~\refm{A''ne}, \refm{A'per},
and~\refm{A'prog}, we have $s\ge 3$, and we proceed to consider separately
the cases where $s=3$, $s=4$, $s=5$, and $s\ge 6$. (The ``typical'' scenario
is addressed in the last case, which also is much less technical to treat;
for this reason, the reader may consider skipping directly to this case.)

\newcases
\case{$s=3$}

By Claim~\refm{A'prog}, the set $\phi_L(A')$ is not an arithmetic
progression; hence, in the representation
  $$ 2A' = 2A_1 \cup 2A_2 \cup 2A_3
                                \cup(A_1+A_2)\cup(A_2+A_3) \cup (A_3+A_1) $$
the union is disjoint and indeed, all sets in the right-hand side reside in
distinct $L$-cosets. (We cannot have $\phi_L(2A_i)=\phi_L(2A_j)$ with $i\ne
j$ since this would imply $2\phi_L(A_i)=2\phi_L(A_j)$, contradicting
rectifiability of $\phi_L(A')$.) Thus,
\begin{multline}
  \frac52(|A_1|+|A_2|+|A_3|) = \frac 52\,|A'| > |2A'| \\
       = |2A_1| + |2A_2| + |2A_3| + |A_1+A_2| + |A_2+A_3| + |A_3+A_1|.
             \label{e:bps3}
\end{multline}

\begin{claim}\label{m:2404a}
$A_1$ is a VSDS; moreover, letting $K:=A_1-A_1$, we have $K\le L$.
\end{claim}

\begin{proof}
Assume for a contradiction that $A_1$ is not a VSDS, and suppose first that
$A_2$ is not a VSDS either. Then $|2A_1|\ge\frac32|A_1|$,
$|2A_2|\ge\frac32|A_2|$, and $|A_1+A_2|\ge|A_2|+\frac12|A_1|$ by
Corollary~\refc{0308a}. Combining these estimates with~\refe{bps3} and the
basic bound
 $|A_i+A_j|\ge|A_i|\ (1\le i\le j\le 3)$, we conclude that
\begin{align*}
  \frac52(|A_1| + |A_2|+|A_3|)
     &> \frac32|A_1| + \frac32|A_2| + |A_3|
                                    + |A_2| + \frac12|A_1| + |A_2| + |A_1| \\
     &= 3|A_1| + \frac72\,|A_2| + |A_3|
\end{align*}
leading to $3|A_3| > |A_1| + 2|A_2|$, a contradiction.

Thus, $A_2$ is a VSDS. Let $k$ denote the number of the $K$-cosets determined
by $A_1$; since $|A_1|\ge|A_2|>\frac32\,|K|$ and $A_1$ is not contained in a
$K$-coset with density exceeding $2/3$, we have $k\ge 2$. Also,
$|2A_1|>\frac32|A_1|$ and $|A_1+A_2|\ge k|A_2|$. This gives
  $$ \frac52(|A_1| + |A_2| + |A_3|)
                   > \frac32|A_1| + |A_2| + |A_3| + k|A_2| + |A_2| + |A_1| $$
whence
  $$ 3|A_3| > (2k-1)|A_2| \ge 3|A_2|, $$
a contradiction. Finally, we notice that $K=A_1-A_1$ implies $K\le L$ (as
$A_1$ is contained in an $L$-coset).
\end{proof}

Let $K$ denote the subgroup of Claim~\refm{2404a}; thus, $A_1$ is contained
in a $K$-coset and $|A_1|>\frac23|K|$.
\begin{claim}
Each of the sets $A_1,A_2,A_3$ is contained in a $K$-coset.
\end{claim}

\begin{proof}
If neither $A_2$ nor $A_3$ is contained in an $K$-coset, then
 $|A_2+A_1|\ge 2|A_1|$ and $|A_3+A_1|\ge 2|A_1|$ whence
   $$ \frac52(|A_1|+|A_2|+|A_3|)
       > |A_1|+|A_2|+|A_3| + 2|A_1| + 2|A_1| + |A_2| $$
resulting in
  $$ 5|A_1| < |A_2| + 3|A_3|, $$
which is wrong.

If $A_2$ is not contained in an $K$-coset, while $A_3$ is, then
$|A_1+A_2|\ge2|A_1|$ and $|A_2+A_3|\ge 2|A_3|$, and then
  $$ \frac52(|A_1|+|A_2|+|A_3|)
       > |A_1|+|A_2|+|A_3| + 2|A_1| + |A_1| + 2|A_3|, $$
  $$ 3|A_1| + |A_3| < 3|A_2|, $$
a contradiction.

Finally, if $A_2$ is contained in an $K$-coset, while $A_3$ is not, then
$|A_1+A_3|\ge2|A_1|$ and $|A_2+A_3|\ge2|A_2|$; as a result,
  $$ \frac52(|A_1|+|A_2|+|A_3|)
       > |A_1|+|A_2|+|A_3| + |A_1| + 2|A_2| + 2|A_1|, $$
  $$ 3|A_1| + |A_2| < 3|A_3|, $$
a contradiction.

The assertion follows.
\end{proof}

Let $A''=B_1\longu B_t$ be the $K$-coset decomposition of $A''$; that is,
each of $B_1\longc B_t$ is contained in a $K$-coset, and the cosets are
pairwise disjoint. Write $\cA':=\phi_K(A')$, $\cA'':=\phi_K(A'')$, and
$\cA:=\phi_K(A)$; thus, $|\cA'|=3$, $|\cA''|=t$, and $|\cA|=3+t$.

We have
  $$ \frac94\,|A| > |2A| \ge |A+A_1|
         \ge (3+t)|A_1| \ge \frac{3+t}3 |A'| > \frac{3+t}3\cdot 0.9|A| $$
whence $t\le 4$. We now improve this estimate as follows.
\begin{claim}\label{m:1807a}
We have $t\le 2$.
\end{claim}

\begin{proof}
Let $\cH:=\pi(\cA+\cA')$. If $|\cA+\cA'|<|\cA|+\frac12|\cA'|$, then by
Lemma~\refl{Olson}, the set $\cA'$ is contained in an $\cH$-coset.
Consequently, $A'$ is contained in a coset of the subgroup
$\phi_K^{-1}(\cH)$. Hence, by Claim~\refm{A'per}, we have
$\phi_K^{-1}(\cH)=\Z_n$; that is, $\cH=\Z_n/K$, meaning that
$\cA+\cA'=\Z_n/K$. Therefore,
 $|\cA+\cA'|=n/|K|\ge |n/L|\ge 37>(3+t)+\frac32=|\cA|+\frac12|\cA'|$, a
contradiction.

We therefore conclude that
\begin{equation}\label{e:1507a}
  |\cA+\cA'| \ge |\cA|+\frac12\,|\cA'|
\end{equation}
and then indeed, rounding to an integer, $|\cA+\cA'|\ge 5+t$. It follows that
the set $A+A'$ consists of the $|\cA|=3+t$ subsets
$2A_1,A_1+A_2,A_1+A_3,A_1+B_1\longc A_1+B_t$, and at least two more subsets
of size at least $|A_3|$ each, all these subsets being pairwise disjoint. As
a result,
\begin{equation}\label{e:0707u}
  |A+A'| \ge (t+3)|A_1| + 2|A_3|.
\end{equation}
On the other hand,
  $$ |A+A'| \le |2A| < \frac94\,|A| < \frac52\,|A'| =
                                           \frac52\,(|A_1|+|A_2|+|A_3|). $$
Comparing this estimate with~\refe{0707u}, we get
\begin{gather*}
  (t+3)|A_1| + 2|A_3| < \frac52\,(|A_1|+|A_2|+|A_3|), \\
  (2t+1)|A_1| < 5|A_2| + |A_3|,
\end{gather*}
whence $t\in\{1,2\}$, as claimed.
\end{proof}

If $|(\cA'+\cA'')\stm 2\cA'|\ge 2$, then $|(B_1+A')\stm(2A')|\ge|A_2|+|A_3|$,
leading to
\begin{multline}\label{e:1207a}
  \frac52\,(|A_1|+|A_2|+|A_3|) = \frac52\,|A'| > \frac94\,|A| \\
       > |2A| \ge |2A'| + (|A_2|+|A_3|) \ge 3|A_1| + 3|A_2| + 2|A_3|.
\end{multline}
On the other hand, from~\refe{bps3} and the trivial estimate
 $|A_i+A_j|\ge |A_i|\ (1\le i\le j\le 3)$,
\begin{equation}
  |2A'| \ge 3|A_1| + 2|A_2| + |A_3|.
\end{equation}
From this estimate and~\refe{1207a} we get $|A_1|+|A_2|<|A_3|$, which is
obviously wrong.

Thus, $|(\cA'+\cA'')\stm 2\cA'|\le 1$. Consequently, for any $\bet\in\cA''$
there are (at least) two elements $\alp\in\cA'$ with $\bet+\alp\in2\cA'$.
Applying Lemma~\refl{alpha} and taking into account that $\cA'$ is not
contained in a four-term progression by Claim~\refm{A'prog}, we conclude that
if $\alp_1,\alp_2\in\cA'$ are elements with $\bet+\alp_i\in2\cA'$, then
$\{\alp_1,\alp_2,\bet\}$ is a coset of the three-element subgroup of
$\Z_n/K$. If $t=1$, then this shows that $A$ is contained in a union of two
cosets of the subgroup of size at most $3|K|$, contradicting
Lemma~\refl{twocoset}. If $t=2$, then writing $\cA''=\{\bet_1,\bet_2\}$, and
applying Lemma~\refl{alpha}, there are elements $\alp,\alp_1,\alp_2\in\cA'$
with $\alp\ne\alp_1$, $\alp\ne\alp_2$ such that both $\{\alp,\alp_1,\bet_1\}$
and $\{\alp,\alp_2,\bet_2\}$ are cosets of the three-element subgroup of
$\Z_n/K$. Sharing the same common element $\alp$, these cosets must be
identical, which is impossible since, for instance, $\alp,\alp_1,\bet,\bet_1$
are pairwise distinct.

\case{$s=4$}

By Claim~\refm{A'prog}, the set $\phi_L(A')$ is not contained in an
arithmetic progression with five or fewer terms; as a result, by
Theorem~\reft{3n-3} (as applied to the set of integers locally isomorphic to
$\phi_L(A')$, with $l=5$), we have
\begin{equation}\label{e:0508c}
  |2\phi_L(A')| \ge 9;
\end{equation}
that is, $2A'$ meets at least nine $L$-cosets. Of these
cosets, four are the cosets determined by the sums $A_1+A_1\longc A_1+A_4$,
and at least five more are determined by some other sums of the form
$A_i+A_j$, with $2\le i\le j\le 4$. Using the trivial estimate
$|A_i+A_j|\ge|A_i|$ for these sums, and observing that in the resulting
estimate the summand $|A_4|$ can appear at most once, and $|A_3|$ at most
twice, we get
\begin{equation}\label{e:2407x}
  \frac52\,|A'| > |2A'| \ge |A_1+A_1|\longp|A_1+A_4|+2|A_2|+2|A_3|+|A_4|.
\end{equation}

\begin{claim}\label{m:2407x}
$A_1$ is a VSDS.
\end{claim}

\begin{proof}
Assuming for the contradiction that $A_1$ is not a VSDS, by
Corollary~\refc{0308a} we have $|A_1+A_2|\ge|A_2|+\frac12\,|A_1|$.
Substituting to~\refe{2407x}, we obtain
\begin{align*}
  \frac52\,|A'| &> \frac32\,|A_1| + \Big(|A_2|+\frac12\,|A_1|\Big)
                                      + 2|A_1|+2|A_2|+2|A_3|+|A_4| \\
    &= 2|A'| + 2|A_1| + |A_2| - |A_4|.
\end{align*}
This simplifies to the obviously wrong inequality
  $$ 3|A_1| + |A_2| < |A_3|+3|A_4|, $$
a contradiction proving the claim.
\end{proof}

Let $K:=A_1-A_1$; thus, $K$ is a subgroup of $L$, and $A_1$ is contained in a
$K$-coset with $|A_1|>\frac23|K|$; also, $|2A_1|=|K|$. Notice that $K$ is
nonzero (else $|A_1|=1$ and then $|A'|=4$ contradicting~\refe{ac1}).

\begin{claim}\label{m:2407y}
Each of the sets $A_1,A_2,A_3,A_4$ is contained in a single $K$-coset.
\end{claim}

\begin{proof}
From~\refe{2407x}, in view of $|2A_1|=|K|$, we have
\begin{gather*}
  \frac52\,|A'| > |K| + 3|A_1|+2|A_2|+2|A_3|+|A_4|, \\
  |A_2| + |A_3| + 3|A_4| > |A_1| + 2|K|,
\end{gather*}
resulting in $|A_1|+3|A_4|>2|K|$. Hence,
  $$ |A_1|+|A_4| = \frac12\,(|A_1|+3|A_4|) + \frac12\,(|A_1|-|A_4|) > |K|, $$
and then indeed $|A_1|+|A_i|>|K|$ for all $i\in[1,4]$, leading to
\begin{equation}\label{e:0508b}
  |A_1| > \frac12\,|K|
\end{equation}
and, by Lemma~\refl{2107a}, to $|A_1+A_i|\ge|K|$.

Substituting this estimate back to~\refe{2407x}, we now get
\begin{gather}
  \frac52\,|A'| > 4|K|+2|A_2|+2|A_3|+|A_4|, \notag \\
  5|A_1| + |A_2| + |A_3| + 3|A_4| > 8|K|, \label{e:2507f}
\end{gather}
which leads to
\begin{gather}
 7|A_1|+3|A_4| > 8|K|, \notag \\
 |A_1|+\frac12\,|A_i| \ge |A_1|+\frac12\,|A_4| > |K|, \label{e:2507b}
\end{gather}
for all $i\in\{2,3,4\}$.

If, for some $i\in\{2,3,4\}$, the set $A_i$ determines two or more
$K$-cosets, then in view of~\refe{0508b} and~\refe{2507b}, by
Lemma~\refl{2107a}~ii) we have $|A_1+A_i|\ge|A_1|+|K|$. Reusing~\refe{2407x},
we then get
\begin{gather*}
  \frac52\,|A'| > 4|K|+|A_1|+2|A_2|+2|A_3|+|A_4|, \\
  3|A_1| + |A_2| + |A_3| + 3|A_4| > 8|K|,
\end{gather*}
which is wrong.
\end{proof}

Notice that from~\refe{2507f} we get
  $$ 8|K| < 5|A_1| + |A_2| + |A_3| + 3|A_4| \le 6|K| + 2(|A_3|+|A_4|). $$
It follows that $|A_i|+|A_j|>|K|$, and therefore $A_i+A_j$ is a $K$-coset for
all $i,j\in[1,4]$ with the possible exception of $i=j=4$. Consequently,
reconsidering the argument that lead us to~\refe{2407x}, we obtain
\begin{equation}\label{e:2507c}
   \frac52\,|A'| > |2A'| \ge 8|K|+|A_4|.
\end{equation}

Let $\cA':=\phi_K(A')$, $\cA'':=\phi_K(A'')$, and $\cA:=\phi_K(A)$; thus
$|\cA'|=4$. Furthermore, from~\refe{0508c} we have
  $$ |2\cA'| = |2\phi_K(2A')| = |\phi_K(2A')|
                                   \ge |\phi_L(2A')| = |2\phi_L(A')| = 9. $$
Indeed, if we had $|2\cA'|\ge 10$, then instead of~\refe{2507c} we would be
able to get the estimate
  $$ \frac52\,|A'| > |2A'| \ge 9|K|+|A_4|, $$
which is wrong in view of $|A'|\le3|K|+|A_4|$. Thus $|2\cA'|=9$. Observing
that $\cA'$ determines $\binom 42+4=10$ sums $\alp_1+\alp_2$ with
$\alp_1,\alp_2\in\cA'$, we conclude that exactly two of these sums coincide,
while the rest are distinct from each other and from the two coinciding sums.

Write $t:=|\cA''|$ and $A''=B_1\longu B_t$ where each of $B_1\longc B_t$ is
contained in a $K$-coset, and the cosets are pairwise distinct; notice that
$|\cA|=4+t$.

If $\cA'+\cA''\not\seq2\cA'$, then there are $i\in[1,4]$ and $j\in[1,t]$ such
that the sum $A_i+B_j$ is disjoint from $2A'$; consequently,
from~\refe{2507c}
\begin{gather*}
  \frac52\,|A'| > \frac94\,(|A'|+|A''|)
               = \frac94\,|A| > |2A| \ge |2A'| + |A_i+B_j|
                                          \ge (8|K|+|A_4|) + |A_4|, \\
  5|A'| > 16|K| + 4|A_4|, \\
  5|A_1| + 5|A_2| + 5|A_3| + |A_4| > 16 |K|
\end{gather*}
which is wrong.

Therefore, $\cA'+\cA''\seq2\cA'$ implying
\begin{equation}\label{e:0508e}
  2\cA = 2\cA' \cup 2\cA''.
\end{equation}
In addition, from $\cA'+\cA''\seq2\cA'$ we derive that $\cA'+\cA\seq2\cA'$,
and since the inverse inclusion holds trivially, we have, indeed,
$\cA+\cA'=2\cA'$. Thus,
  $$ |\cA'|=4,\ |\cA''|=t,\ |\cA|=4+t,\ |\cA+\cA'|=|2\cA'| = 9. $$

From $A_1+A_1\longc A_1+A_4,A_1+B_1\longc A_1+B_t\seq 2A$ we get
  $$ \frac94\,|A| > |2A| \ge (t+4)|A_1|
                  \ge \frac{t+4}4\,|A'| > 0.9\, \frac{t+4}4\,|A| $$
which yields $t\le 5$. We can improve this bound as follows.

\begin{claim}\label{m:1807a}
We have $t\le 3$.
\end{claim}

\begin{proof}
Let $\cH:=\pi(\cA+\cA')$. If $|\cA+\cA'|<|\cA|+\frac12|\cA'|$, then by
Lemma~\refl{Olson}, the set $\cA'$ is contained in an $\cH$-coset.
Consequently, $A'$ is contained in a coset of the subgroup
$\phi_K^{-1}(\cH)$. Hence, by Claim~\refm{A'per}, we have
$\phi_K^{-1}(\cH)=\Z_n$; that is, $\cH=\Z_n/K$, meaning that
$\cA+\cA'=\Z_n/K$. Therefore,
 $|\cA+\cA'|=n/|K|\ge |n/L|\ge 37>6+t=|\cA|+\frac12|\cA'|$, a
contradiction.

Thus, $|\cA+\cA'|\ge|\cA|+\frac12|\cA'|=t+6$ showing that the set $A+A'$
consists of the $|\cA|=4+t$ subsets
$2A_1,A_1+A_2,A_1+A_3,A_1+A_4,A_1+B_1\longc A_1+B_t$, and at least two more
subsets of size at least $|A_4|$ each (with all these subsets pairwise
disjoint). As a result,
  $$ |A+A'| \ge (t+4)|A_1| + 2|A_4|. $$
On the other hand,
  $$ |A+A'| \le |2A| < \frac94\,|A| < \frac52\,|A'| =
                             \frac52\,(|A_1|+|A_2|+|A_3|+|A_4|). $$
Comparing the last two estimates, we get
  $$ (2t+3)|A_1| < 5|A_2|+5|A_3|+|A_4| $$
whence $t\le 3$.
\end{proof}

\subcase{$t=1$} In this case we have $A=A'\cup A''$ where $A'=A_1\cup A_2\cup
A_3\cup A_4$ with $A_1\longc A_4$ residing in pairwise distinct $K$-cosets,
and $A''$ resides in yet another $K$-coset. Moreover, $2A'$ is a disjoint
union of eight $K$-cosets, and one more set which is either a $K$-coset, or
the set $2A_4$ (contained in a $K$-coset). Also, from~\refe{0508e}, there are
at most two $K$-cosets containing some elements of $2A$, but not entirely
contained in $2A$: namely, the cosets determined by $2A_4$ and $2B''$. It
follows that $|2A+K|-|2A|\le(|K|-|2A_4|)+(|K|-|2A''|)$. Also,
$|A+K|-|A|=5|K|-|A|$. On the other hand, we observe that $K$ is nonzero (as
otherwise we would have $|A|=|\cA|=5$ contradicting~\refe{ac1}), and that
$|2A+K|\ne\Z_n$ (otherwise $\frac n{|K|}=|2\cA|\le|2\cA'|+1=10$ while, on the
other hand, $\frac n{|K|}\ge\frac n{|L|}\ge 37$). Consequently, we can apply
Lemma~\refl{MofA} to get
\begin{gather*}
  (|K|-|2A_4|)+(|K|-|2A''|) > 5|K|-|A|, \\
  |A| > 3|K| + |2A_4| + |2A''|
\end{gather*}
which is wrong in view of
 $$|A|=|A_1|+|A_2|+|A_3|+|A_4|+|A''| \le 3|K| +|A_4|+|A''|. $$

\bigskip
\subcase{$t\in\{2,3\}$} In this case $|\cA'|=4$, $|\cA''|=t$, $|\cA|=4+t$,
and $|\cA+\cA'|=|2\cA'|=9=|\cA|+|\cA'|-(t-1)$. Furthermore,
$|\cA|+|\cA'|=9<{|\Z_n/L|}\le{|\Z_n/K|}$, $|\cA|\ge|\cA'|\ge 2$, and $\cA'$
is rectifiable, not an arithmetic progression (by Claim~\refm{A'prog}), and
not contained in a proper coset (as a consequence of Claim~\refm{A'per}).
Thus, the assumptions of Lemma~\refl{kemp} are satisfied. Applying the lemma,
we conclude that there is a nonzero, finite, proper subgroup $\cH<\Z_n/K$
such that $\cA'$ meets two $\cH$-cosets and has exactly $(|H|+1)/2$ elements
in each of them. Since $|\cA'|=4$, we have $|\cH|=3$; thus, we can write
$\cA'=\{\alp_1,\alp_1+\del,\alp_2,\alp_2+\del\}$ where $\del$ a fixed element
of the group $\Z_n/K$ of order $3$ (so that $\cH=\{0,\del,2\del\}$), and
$\alp_1,\alp_2\in\Z_n/K$ belong to distinct $\cH$-cosets. Notice that, since
$\cA'$ is rectifiable, the element $\alp_1-\alp_2$ has order at least $4$.

Fix arbitrarily $\bet\in\cA''$. From $\cA'+\cA''\seq2\cA'$
(cf.~\refe{0508e}), we have $\bet+\alp\in2\cA'$ for any $\alp\in\cA''$. By
inspection, there are two elements $\bet\notin\cA'$ with this property:
$\bet=\alp_1-\del$ and $\bet=\alp_2-\del$. Hence, $t=2$ and
$\cA''=\{u-\del,v-\del\}$. It follows that
$\cA=(\alp_1+\cH)\cup(\alp_2+\cH)$; consequently, $A$ is contained in the
union of two cosets of the subgroup $\psi_K^{-1}(\cH)$. Since this subgroup
has size at most $|K||\cH|=3|K|\le 3|L|<n/2$, we can invoke
Lemma~\refl{twocoset} to complete the proof.

\case{$s=5$}

By Claim~\refm{A'prog}, the set $\phi_L(A')$ is not contained in an
arithmetic progression with seven or fewer terms; as a result, by
Theorem~\reft{3n-3} (as applied to the set of integers locally isomorphic to
$\phi_L(A')$, with $l=7$), we have
\begin{equation}\label{e:2907a}
  |2\phi_L(A')| \ge 12;
\end{equation}
that is, $2A'$ meets at least twelve $L$-cosets. Of these cosets, five are
the cosets determined by the sums $A_1+A_1\longc A_1+A_5$, and at least seven
more are determined by some other sums of the form $A_i+A_j$, with $2\le i\le
j\le 5$. Using the trivial estimate $|A_i+A_j|\ge|A_i|$ for these sums, and
observing that in the resulting estimate the summand $|A_5|$ can appear at
most once, $|A_4|$ at most twice, and $A_3$ at most three times, we get
  $$ \frac52\,|A'| > |2A'|
                   \ge |A_1+A_1|\longp|A_1+A_5|+|A_2|+3|A_3|+2|A_4|+|A_5|. $$
As a result
\begin{multline}\label{e:2807d}
  \frac52\,|A'| > |2A'| \ge 5|A_1|+|A_2|+3|A_3|+2|A_4|+|A_5| \\
       = 2|A'| + 3|A_1| - |A_2| + |A_3| - |A_5|.
\end{multline}
It follows that
\begin{equation}\label{e:1607a}
  5|A_1| + |A_3| < 3|A_2|+|A_4|+3|A_5|;
\end{equation}
consequently, $5|A_1|<3|A_2|+3|A_5|$ resulting in
  $$ |A_2| > \frac56\,|A_1|. $$

\begin{claim}\label{m:2107a}
$A_1$ is a VSDS.
\end{claim}

\begin{proof}
If $|2A_1|\ge\frac32\,|A_1|$, then the summand $5|A_1|$ in~\refe{2807d} can
be replaced with $\frac{11}2\,|A_1|$, and then~\refe{1607a} can be improved
to $6|A_1|+|A_3|<3|A_2|+|A_4|+3|A_5|$. However, this implies
$6|A_1|<3|A_2|+3|A_5|$ which is obviously wrong.
\end{proof}

With Claim~\refm{2107a} in mind, let $K:=A_1-A_1$; thus, $K\le L$ is a
subgroup, $A_1$ is contained in a $K$-coset, $|A_1|>\frac23|K|$, and
$|2A_1|=|K|$. Notice that $K$ is nonzero (else $|A_1|=1$ and then $|A'|=5$
contradicting~\refe{ac1}).

From~\refe{1607a} we get
  $$ 5|A_1| < 3|A_2|+3|A_5| \le 3|A_1| + 3|A_5| $$
whence $|A_i|\ge|A_5|>\frac23\,|A_1|$ for each $i\in[1,5]$. Therefore
$|A_1|+|A_i|\ge\frac53\,|A_1|>|K|$, and then $|A_1+A_i|\ge|K|$ by
Lemma~\refl{2107a}. Consequently, we can improve~\refe{1607a} to write
\begin{multline*}
  \frac52\,|A'| > |2A'| \ge 5|K|+|A_2|+3|A_3|+2|A_4|+|A_5| \\
                          = 2|A'| + 5|K| - 2|A_1| - |A_2| + |A_3| - |A_5|.
\end{multline*}
It follows that
\begin{gather*}
  |A'| > 10|K| - 4|A_1| - 2|A_2| + 2|A_3| - 2|A_5|, \\
  5|A_1| + 3|A_2| + |A_4| + 3|A_5| > 10|K| + |A_3|, \\
  10|K| < 5|A_1| + 3|A_2| + 3|A_5| < 8|K| + 3|A_5|,
\end{gather*}
implying
\begin{equation}\label{e:1807a}
  |A_2|\ge\dotsb\ge|A_5|>\frac23|K|.
\end{equation}
Therefore
\begin{equation}\label{e:0608a}
  |A_i|+2|A_1|>2|K|.
\end{equation}

\begin{claim}\label{m:sep}
Each of the sets $A_1\longc A_5$ is contained in a single $K$-coset.
\end{claim}

\begin{proof}
By Lemma~\refl{2107a}, from~\refe{0608a} it follows that if, for an index
$i\in[2,5]$, the set $A_i$ meets two or more $K$-cosets, then $|A_1+A_i|\ge
|K|+|A_1|$. Hence, in this case
\begin{multline*}
  \frac52\,|A'| > (5|K|+|A_1|) + |A_2|+3|A_3|+2|A_4|+|A_5| \\
                            = 5|K| + 2|A'| - |A_1| -|A_2| + |A_3| - |A_5|,
\end{multline*}
leading to
  $$ 3|A_1| + 3|A_2| + |A_4| + 3|A_5| > 10|K|+|A_3| $$
which is wrong as the sum in the left-hand side is at most $9|K|+|A_4|$.
\end{proof}

As it follows from Claim~\refm{sep} and~\refe{1807a}, we have $|A_i+A_j|=|K|$
for all $i,j\in[1,5]$. Hence, $2A'$ is $K$-periodic and
  $$ |2A'| \ge 12|K| $$
(cf.~\refe{2907a}); indeed, equality holds as $|A'|\le 5|K|$ implies
$|2A'|<\frac52\,|A'|<13|K|$.

Let $\cA':=\phi_K(A')$, $\cA'':=\phi_K(A'')$, and $\cA:=\phi_K(A)$; thus
$|\cA'|=5$ and $|2\cA'|=12$. Also, write $t:=|\cA''|$ and $A''=B_1\longu B_t$
where each of $B_1\longc B_t$ is contained in a $K$-coset and the cosets are
pairwise distinct; notice that $|\cA|=5+t$.

If $\cA'+\cA''\not\seq2\cA'$, then there are $i\in[1,5]$ and $j\in[1,t]$ such
that the sum $A_i+B_j$ is disjoint from $2A'$; consequently,
\begin{gather*}
  \frac52\,|A'| > \frac94\,(|A'|+|A''|)
               = \frac94\,|A| > |2A| \ge |2A'| + |A_i+B_j|
                                          \ge 12|K|+|A_5|, \\
  5|A'| > 24|K| + 2|A_5|, \\
  5|A_1| + 5|A_2| + 5|A_3| + 5|A_4| + 3|A_5| > 24|K|
\end{gather*}
which is wrong.

Therefore, $\cA'+\cA''\seq2\cA'$; as a result, $\cA+\cA'\seq2\cA'$, and since
the inverse inclusion is trivial, we have, indeed, $\cA+\cA'=2\cA'$.

The relation $\cA'+\cA''\seq 2\cA'$ also shows that
$2\cA=(2\cA')\cup(2\cA'')$. Since $2A$ is aperiodic by Lemma~\refl{2Aper},
while $2A'$ is $K$-periodic as a consequence of~\refe{1807a}, we conclude
that there exist $i,j\in[1,t]$ such that $B_i+B_j$ is disjoint from $2A'$.

\begin{claim}
We have $t\le 4$.
\end{claim}

\begin{proof}
Arguing as in the proof of Claim~\refm{1807a}, from Lemma~\refl{Olson} we
obtain
  $$ |\cA+\cA'| \ge |\cA|+\frac12\,|\cA'| \ge (5+t) + \frac52 $$
(cf.~\refe{1507a}). Thus, the set $A+A'$ consists of the $|\cA|=5+t$ subsets
$2A_1,A_1+A_2,A_1+A_3,A_1+A_4,A_1+A_5,A_1+B_1\longc A_1+B_t$, and at least
$\lcl\frac52\rcl=3$ more subsets of size at least $|A_5|$ each (with all
these subsets pairwise disjoint). As a result,
  $$ |A+A'| \ge (t+5)|A_1| + 3|A_5|. $$
On the other hand,
  $$ |A+A'| \le |2A| < \frac94\,|A| < \frac52\,|A'| =
                             \frac52\,(|A_1|+|A_2|+|A_3|+|A_4|+|A_5|). $$
Comparing the last two estimates, we get
  $$ (2t+5)|A_1| + |A_5| < 5|A_2|+5|A_3|+5|A_4| $$
whence $t\le 4$.
\end{proof}

\subcase{$t=1$} As explained above, in this case $2B_1$ is disjoint from
$2A'$. As a result, $|2A|\ge|2A'|+|2B_1| \ge 12|K|+|A''|$ and then
\begin{gather}
  \frac94\,(|A'|+|A''|) = \frac94\,|A| > |2A| \ge 12|K|+|A''|, \label{e:1707a} \\
  9|A'| + 5|A''| > 48|K|, \notag \\
  48|K| < \frac{86}9\,|A'| \le \frac{430}{9}\,|K| \notag
\end{gather}
which is wrong.

\subcase{$t=2$} Write $\bet_i:=\phi_K(B_i)$, $i\in\{1,2\}$; thus,
$\cA''=\{\bet_1,\bet_2\}$. Since $2\cA''\not\seq2\cA'$, there is a pair of
indices $1\le i\le j\le 2$ such that $\bet_i+\bet_j\notin2\cA'$. Suppose
first that $(i,j)$ is the unique pair with this property. In this situation
we have $|2A+K|-|2A|=|K|-|B_i+B_j|$ and $|A+K|-|A|=7|K|-|A|$. On the other
hand, $K$ is nonzero (as otherwise we would have $|A|=|\cA|=7$), and
$|2A+K|\ne\Z_n$ (otherwise $\frac n{|K|}=|2\cA|\le|2\cA'|+\binom t2+t=15$
while, on the other hand, $\frac n{|K|}\ge\frac n{|L|}\ge 37$). Consequently,
$|K|-|B_i+B_j|>7|K|-|A|$ by Lemma~\refl{MofA}, which yields
  $$ |A| > 6|K| + |B_i+B_j|. $$
From this estimate and
  $$ |A| = |A'| + |A''| \le 5|K| + |B_1|+|B_2| $$
we get $|B_1|+|B_2|>|B_i+B_j|+|K|$, which is impossible in view of
$\max\{|B_1|,|B_2|\}<|K|$  and $\min\{|B_1|,|B_2|\}\le|B_i+B_j|$.

We therefore conclude that there are at least two pairs $(i,j)$ with $1\le
i\le j\le 2$ and $\bet_i+\bet_j\notin\cA'$. If, moreover, one can find two
such pairs so that the sums $\bet_i+\bet_j$ are distinct from each other,
then the two corresponding sumsets $B_i+B_j$ jointly contain at least
$|B_1|+|B_2|=|A''|$ elements (which may not be obvious, but is not difficult
to see either). Consequently,
  $$ |2A| \ge |2A'|+|A''| \ge 12|K| + |A''| $$
leading to a contradiction as in the case $t=1$, cf.~\refe{1707a}.

We are left with the case where there are at least two pairs of indices $1\le
i\le j\le 2$ with $\bet_i+\bet_j\notin2\cA'$, but the sums $\bet_i+\bet_j$
are equal to each other for all such pairs $(i,j)$. Since $\bet_1+\bet_2$ is
distinct from each of $2\bet_1$ and $2\bet_2$, we actually have
$2\bet_1=2\bet_2$; that is, the two pairs are $(1,1)$ and $(2,2)$, while
$\bet_1+\bet_2\in2\cA'$. Acting as above, we get in this case
$|2A+K|-|2A|=|K|-|2B_1\cup2B_2|$ and $|A+K|-|A|=7|K|-|A|$, whence
$|K|-|2B_1\cup2B_2|>7|K|-|A|$ by Lemma~\refl{MofA}. Therefore $|A|>
6|K|+|2B_1\cup2B_2|$ which, along with $|A| = |A'| + |A''| \le 5|K| +
|B_1|+|B_2|$, gives $|B_1|+|B_2|>|2B_1\cup2B_2|+|K|$. This is impossible in
view of $\max\{|B_1|,|B_2|\}\le\min\{|K|,|2B_1\cup2B_2|\}$.

\subcase{$t\in\{3,4\}$} In this case $|\cA'|=5$, $|\cA''|=t$, $|\cA|=5+t$,
and $|\cA+\cA'|=|2\cA'|=12=|\cA|+|\cA'|-(t-2)$. Furthermore,
$|\cA|+|\cA'|=12<|\Z_n|/|K|$, $|\cA|\ge|\cA'|\ge 2$, and $\cA'$ is
rectifiable, not an arithmetic progression (by Claim~\refm{A'prog}) and not
contained in a proper coset (as a consequence of Claim~\refm{A'per}). Thus,
the assumptions of Lemma~\refl{kemp} are satisfied. Applying the lemma, we
conclude that $|\cA'|$ is even, a contradiction.

\case{$s\ge 6$} In this case $\tau':=|2A'|/|A'|<\frac52\le3(1-1/s)$. In view
of this estimate, and since $\phi_L(A')$ is contained in a rectifiable subset
of $\Z_n/L$, we can apply Proposition~\refp{combo} to the set $A'$ to find a
proper subgroup $H'<\Z_n$ and a progression $P'\seq\Z_n$ such that $A'\seq
P'+H'$, $|P'+H'|=|P'||H'|$, and $(|P'|-1)|H'|\le|2A'|-|A'|$.

By Claim~\refm{A'per} and Lemma~\refl{1.5}, we have
\begin{equation}\label{e:0102a}
  |2A'|\ge\frac32|A'|.
\end{equation}

If $A$ contained an element $a\notin(2P'-P')+H'$, then $a+A'\seq a+P'+H'$
would be disjoint from $2A'\seq 2P'+H'$, and in view of~\refe{0102a} we would
get
  $$ |2A| \ge |a+A'|+|2A'| \ge \frac52\,|A'| > \frac94\,|A|, $$
contradicting the small-doubling assumption. Thus, $A$ is entirely contained
in the set $(2P'-P')+H'$:
\begin{equation}\label{e:2P'P'}
  A \seq 2P'-P'+H'.
\end{equation}

Let $d$ denote the difference of the arithmetic progression $2P'-P'$. Since
$A$ is contained in a coset of the subgroup generated by $d$ and $H'$, this
subgroup is not proper; that is, the order of $\phi_{H'}(d)$ in the quotient
group $\Z_n/H'$ is $m':=n/|H'|$.

On the other hand, from~\refe{0308a} and Lemma~\refl{numbers},
\begin{multline*}
  |2P'-P'| = 3|P'|-2 = 3(|P'|-1) + 1 \\
     \le \frac3{|H'|}(|2A'|-|A'|) + 1
       = \frac3{|H'|}\Big(1-\frac1{\tau'}\Big)\,|2A'| + 1
         < \frac9{5|H'|}\,|2A|+1 \\
             < \frac9{5|H'|}\cdot 2C_0^{-1}n+1 < \frac{m'}2 + 1.
\end{multline*}

Thus, $\phi_{H'}(2P'-P')$ is an arithmetic progression with the difference
generating $\Z_n/H'$, and of size not exceeding $(|\Z_n/H'|+1)/2$; hence, a
rectifiable set. In view of~\refe{2P'P'}, the set $\phi_{H'}(A)$ is
rectifiable, too. Also, since $A$ meets at least four $H'$-cosets,
  $$ |2A| < \frac94\,|A| \le 3\Big(1-\frac1{|\phi_{H'}(A)|}\Big) |A|. $$
Consequently, we can apply Proposition~\refp{combo} to find a proper subgroup
$H<\Z_n$ and a progression $P\seq\Z_n$ such that $A\seq P+H$, $|P+H|=|P||H|$,
and $(|P|-1)|H|\le|2A|-|A|$. Thus $A$ is regular, contrary to the choice of
$A$ as a counterexample set.

This completes the proof in the case $s\ge 6$.

%\newpage

\bigskip

\end{document}